\documentclass[12pt]{amsart}
\usepackage[textwidth=16cm, textheight=22cm]{geometry}
\usepackage{amssymb}
\usepackage{amsmath}
\usepackage{latexsym}
\usepackage{tikz}
\newtheorem{theorem}{Theorem}[section]

\newtheorem{proposition}[theorem]{Proposition}

\newtheorem{corollary}[theorem]{Corollary}
\newcommand{\msn}{\medskip\noindent}

\newcommand{\RR}{\mathbb{R}}
\newcommand{\R}{\Rp}
\newcommand{\Rnn}{\R^{n\times n}}
\newcommand{\Rp}{\RR_+}

\newcommand{\Rpnn}{\Rp^{n\times n}}
\newcommand{\Rmn}{\R^{m\times n}}

\newcommand{\Rn}{\R^n}

\newcommand{\Nca}{N_c(A)}
\newcommand{\Eca}{E_c(A)}

\newcommand{\spann}{\operatorname{span}}
\newcommand{\diag}{\operatorname{diag}}
\newcommand{\bunity}{1}
\newcommand{\bzero}{0}
\newcommand{\blzero}{{\mathbf 0}}

\newcommand{\caa}{\gamma}

\newcommand{\Attr}{\operatorname{Attr}}

\newcommand{\modd}{\operatorname{mod}}

\newcommand{\bul}{{\mathcal S}}
\newcommand{\bulnn}{\bul^{n\times n}}
\newcommand{\lcm}{\operatorname{lcm}}
\def\Core{Core}
\def\Digr{{\mathcal D}}
\def\crit{{\mathcal C}}
\def\critgraph{{\mathcal C}}
\def\critnodes{c}
\def\cyclasses{\Tilde{c}}
\def\noncrit{\overline{c}}
\def\critcomps{n_c}

\begin{document}
\title{Cyclic classes and attraction cones in max algebra}
\author{Serge\u{\i} Sergeev}
\address{University of Birmingham,
School of Mathematics, Watson Building, Edgbaston B15 2TT, UK}
\email{sergiej@gmail.com}
\thanks{This research was supported by EPSRC grant RRAH12809, 
RFBR grant 08-01-00601 and RFBR/CNRS grant 05-01-02807}

\begin{abstract}
In max algebra it is well-known that the sequence $A^k$,
with $A$ an irreducible square matrix, becomes periodic at
sufficiently large $k$. This raises a number of questions
on the periodic regime of $A^k$ and $A^k\otimes x$, for a given vector $x$.
Also, this leads to the concept of attraction cones in max algebra,
by which we mean sets of vectors with ultimate orbit period not
exceeding a given number.

This paper shows that some of these questions can be solved by matrix
squaring ($A$,$A^2$,$A^4$, ...), analogously to recent
findings of Seman\v{c}\'{\i}kov\'a \cite{Sem-06,Sem-07} concerning
the orbit period in max-min algebra. Hence the computational complexity
of such problems is $O(n^3\log n)$. The main idea is to apply an appropriate
diagonal similarity scaling $A\mapsto X^{-1}AX$,
called visualization scaling, and to study the role of
cyclic classes of the critical graph.

For powers of a visualized matrix in the
periodic regime, we observe remarkable symmetry described by
circulants and their rectangular generalizations. We exploit this symmetry to
derive a system of equations for attraction cone, and
present an algorithm which computes the coefficients of the system.
\end{abstract}
\maketitle

\section{Introduction}

By \textit{max algebra} we understand the analogue of linear algebra
developed over the max-times semiring $\RR_{\max,\times}$ which is
the set of nonnegative numbers $\RR_+$ equipped with the operations
of ``addition'' $a\oplus b:=\max(a,b)$ and the ordinary
multiplication $a\otimes b:=a\times b$. 
Zero and unity of this semiring coincide
with the usual $0$ and $1$. 
The operations of the
semiring are extended to the nonnegative matrices and vectors in the
same way as in conventional linear algebra. That is if $A=(a_{ij})$,
$B=(b_{ij})$ and $C=(c_{ij})$ are matrices of compatible sizes with
entries from $\R$,
we write $C=A\oplus B$ if $c_{ij}=a_{ij}\oplus b_{ij}$ for all $i,j$ and $%
C=A\otimes B$ if $c_{ij}=\sum_{k}^{\oplus }a_{ik}
b_{kj}=\max_{k}(a_{ik} b_{kj})$ for all $i,j$. 
If $A$ is a square matrix over $\R$ then the
iterated product $A\otimes A\otimes
...\otimes A$ in which the symbol $A$ appears $k$ times will be denoted by $%
A^{k}$.

The {\em max-plus semiring}
$\RR_{\max,+}=(\RR\cup\{-\infty\}, \oplus=\max,\otimes=+)$,
developed over the set of real numbers $\RR$ with adjoined
element $-\infty$ and the ordinary addition playing the
role of multiplication, is another
isomorphic ``realization'' of max algebra.  In
particular, $x\mapsto\exp(x)$ yields an isomorphism between
$\RR_{\max,+}$ and $\RR_{\max,\times}$. In the max-plus setting,
the zero element is $-\infty$ and the unity is $0$. 

The {\em min-plus semiring}
$\RR_{\min,+}=(\RR\cup\{+\infty\}, \oplus=\min,\otimes=+)$
is also isomorphic to $\RR_{\max,+}$ and
$\RR_{\max,\times}$. Another well-known semiring
is the {\em max-min semiring} 
$\RR_{\max,\min}=(\RR\cup\{-\infty\}\cup\{\infty\},\oplus=\max,\otimes=\min)$,
see \cite{Gav:04,Sem-06,Sem-07},
but it is {\em not} isomorphic to any of the semirings above.

Max algebraic column spans of nonnegative matrices $A\in\Rpnn$ are
sets of max linear combinations of columns $\bigoplus_{i=1}^n \alpha_i A_{\cdot i}$
with nonnegative coefficients $\alpha_i$. Such column spans are
{\em max cones}, meaning that they are closed under componentwise maximum
$\oplus$ and multiplication by nonnegative scalars. There are important
analogies and links between max cones and convex cones
\cite{CGQS-05,DS-04,SSB,Ser-08}.

The maximum cycle geometric mean $\lambda(A)$, see below
for exact definition, is one of the most 
important charasteristics of a matrix $A\in\Rnn$
in max algebra. In particular, it is the largest eigenvalue
of the spectral problem $A\otimes x=\lambda x$. The cycles at which this
maximum geometric mean is attained, are called {\em critical}. Further,
one consideres the {\em critical graph} $\critgraph(A)$ 
which consists of all nodes and edges that belong
to the critical cycles. This graph is crucial for the description
of eigenvectors \cite{BCOQ,CG:79,HOW:05}.
 
The well-known {\em cyclicity theorem} states that if $A$ is irreducible, then 
the sequence $A^k$ becomes periodic after some finite transient time, and
that the ultimate period of $A^k$ is equal to the cyclicity of the critical
graph \cite{BCOQ,CG:79,HOW:05}. Generalizations to reducible
case, computational complexity issues and important special cases
of this result have
been extensively studied in \cite{BdS,Gav-00,Gav:04,Mol-05,MP-00}.

In this paper we study the behaviour of matrix powers and orbits
$A^k\otimes x$ in the irreducible
case in the periodic regime, i.e., after the periodicity is reached. 
One of the main ideas is to
study the periodicity of {\em visualized} matrices, 
meaning matrices with all entries less than or equal to the maximum cycle
geometric mean. This study provides
a connection to the theory of Boolean matrices \cite{BR,Kim:82}.

In Boolean matrix algebra, one considers components
of imprimitivity of a matrix \cite{BR,Kim:82}, or equivalently, 
cyclic classes of the associated digraph \cite{BV-73}.
In max algebra, cyclic classes of the critical graph have been
considered as an important tool in the proof of the cyclicity theorem
mentioned above, see \cite{HOW:05} Sect.~3.1. 
Recently, the cyclic classes appeared 
in max-min algebra \cite{Sem-06,Sem-07}, where they were
used to study the ultimate periods of orbits and other periodicity problems. 
It was shown that such questions can be solved by matrix squaring ($A$, $A^2$,
$A^4$, $A^8$, ...), which yields computational
complexity $O(n^3\log n)$. 

We show that the problems of computing 
ultimate period and matrix powers
in the periodic regime can be solved by matrix squaring in {\em max algebra},
which yields the same complexity bound $O(n^3\log n)$. This is achieved by exploiting
visualization, and cyclic classes of the critical graph.
Further it turns out that the periodic powers of visualized matrices
have remarkable symmetry described by circulant matrices and their
rectangular generalizations. We use this symmetry to derive a
system of equations for {\em attraction cone}, meaning
the max cone which consists of all vectors $x$ whose ultimate 
period of $A^k\otimes x$ does not exceed a given number.
We also describe extremals of attraction cones
and present an algorithm for computing the coefficients of this
system in the case when $\crit(A)$ is strongly connected.

The contents of the paper are as follows. In Section 2 we revise two important
topics in max algebra, namely the spectral problem and Kleene stars. In Section 3,
we speak of the visualization and the connection to the theory of Boolean matrices
which it provides, see Propositions \ref{vis-exist} and \ref{butkovic}.
In Section 4, we study basic properties of matrix powers
in the periodic regime, see Propositions
\ref{p:lindep} -- \ref{p:ultspan}. The problems
which can be solved by matrix squaring are described in
Theorem \ref{blanka}. In Section 5 we
introduce some useful constructions associated with irreducible
visualized matrices and their powers, namely, core matrix,
$CSR$-representation, and describe their circulant simmetries in 
Proposition~\ref{p:critcirc}.
In Section 6 we derive a concise system for attraction cone,
see Theorem~\ref{mainres}. We
also describe extremals and present an algorithm for computing the coefficients of this
system in the case when $\crit(A)$ is strongly connected.
We conclude with Section 7 which is devoted to numerical examples.

As $\RR_{\max,+}$ and $\RR_{\max,\times}$ are isomorphic, we use 
the possibility
to switch between them, but only when it is really convenient.
Thus, while the theoretical results are obtained over
max-times semiring, which looks more natural in connection with diagonal
matrix scaling and Boolean matrices,
the examples in Section 7 are written over {\em max-plus semiring},
where it is much easier to calculate.

We remark that some aspects of the theory
of attraction spaces have been
investigated in \cite{Bra:93,Dok:08,Mai-95} in certain
special cases. 
Also, the periodicity of max algebraic powers of matrices
can be regarded from the viewpoint of max-plus semigroups as studied
in \cite{Mer-09}.

\section{Two topics in max algebra}

\subsection{Spectral problem}
Let $A\in\Rnn$. Consider the problem of finding
$\lambda\in\R$ and nonzero $x\in\Rn$ such that
\begin{equation}
\label{spectral0}
A\otimes x=\lambda x.
\end{equation}
If for some $\lambda$ there exists a nonzero $x\in\Rn$ which satisfies \eqref{spectral0},
then $\lambda$ is called a {\em max-algebraic eigenvalue} of $A$, and $x$ is a {\em max-algebraic
eigenvector} of $A$ associated with $\lambda$. 
With the zero vector adjoined,
the set of max-algebraic eigenvectors associated
with $\lambda$ forms a max cone, which is called the {\em eigencone} 
associated with $\lambda$.

The largest max-algebraic eigenvalue of $A\in\Rnn$ is equal to
\begin{equation}
\label{lambdatrace}
\lambda(A)=\bigoplus_{k=1}^n (\operatorname{Tr}_{\oplus} A^k)^{1/k},
\end{equation}
where $\operatorname{Tr}_{\oplus}$ is defined by 
$\operatorname{Tr}_{\oplus}(A):=\bigoplus_{i=1}^n a_{ii}$
for any $A=(a_{ij})\in\Rnn$.
Further we explain the graph-theoretic meaning of \eqref{lambdatrace},
assumed that $\lambda(A)\neq\bzero$.

With $A=(a_{ij})\in\Rnn$ we can associate the weighted digraph $\Digr(A)=(N(A),E(A))$,
with the set of nodes $N(A)=\{1,\ldots,n\}$ and the set of edges 
$E(A)=\{(i,j)\mid a_{ij}\neq\bzero\}$ with weights $w(i,j)=a_{ij}$.
Suppose that $\pi =(i_{1},...,i_{p})$ is a path in $\Digr(A)$, then the \textit{%
weight} of $\pi $ is defined to be $w(\pi
,A)=a_{i_{1}i_{2}}a_{i_{2}i_{3}}\ldots a_{i_{p-1}i_{p}}$ if $p>1$, and $1$ if $p=1$.
If $i_1=i_p$ then $\pi$ is called a cycle.
One can check that
\begin{equation*}
\lambda(A)=\max_{\sigma }\mu (\sigma ,A),  \label{mcm}
\end{equation*}%
where the maximization is taken over all cycles in $\Digr(A)$ and
\begin{equation*}
\mu (\sigma ,A)=w(\sigma ,A)^{1/k}  \label{cm}
\end{equation*}%
denotes the \textit{geometric mean} of the cycle $\sigma =(i_{1},...,i_{k},i_{1})$.
Thus $\lambda(A)$ is the {\em maximum cycle geometric mean} of $\Digr(A)$.

$A\in\Rnn$ is {\em irreducible} if for any nodes $i$ and $j$ there exists a path 
in $\Digr(A)$,
which begins at $i$ and ends at $j$. In this case $A$ has a unique max-algebraic eigenvalue
which equals $\lambda(A)$.

Note that $\lambda(\alpha A)=\alpha\lambda(A)$ and hence 
$\lambda(A/\lambda(A))=\bunity$ if $\lambda(A)>0$. 
Unless we need matrices with $\lambda(A)=\bzero$,
we can always assume without loss of generality that 
$\lambda(A)=\bunity$. Such matrices will be called
{\em definite}.

An important relaxation of \eqref{spectral0} is
\begin{equation}
\label{subeigs0}
A\otimes x\leq \lambda x.
\end{equation}
The nonzero vectors $x\in\Rn$ which satisfy \eqref{subeigs0} are 
called {\em subeigenvectors} associated
with $\lambda$. With the zero vector adjoined, 
they form a max cone 
called {\em subeigencone}.
This is a conventionally convex cone, 
meaning that it is closed under the {\em ordinary} addition.
See \cite{SSB} for more details.

The eigencone (resp. subeigencone) of $A$ associated 
with $\lambda(A)$ will be denoted by
$V(A)$ (resp. $V^*(A)$).  

\subsection{Kleene stars} 
Let $A\in\Rnn$. Consider the formal series
\begin{equation}
\label{kleene0}
A^*=I\oplus A\oplus A^2\oplus\ldots,
\end{equation}
where $I$ denotes the identity matrix
with entries
\begin{equation*}
\delta_{ij}=
\begin{cases}
\bunity, &\text{if $i=j$,}\\
\bzero, &\text{otherwise.}
\end{cases}
\end{equation*}
Series \eqref{kleene0} is a max-algebraic analogue of $(I-A)^{-1}$, and
it converges to a matrix with finite entries if and only if $\lambda(A)\leq\bunity$
\cite{BCOQ,Car-71}. In this
case
\begin{equation}
\label{kleene1}
A^*=I\oplus A\oplus A^2\oplus\ldots\oplus A^{n-1},
\end{equation}
which is called the {\em Kleene star} of $A$.

For any $A\in\Rnn$,
\begin{equation}
\label{kleene2}
\text{$A$ is a Kleene star}\ \Leftrightarrow\ A^2=A,\ a_{ii}=\bunity\;\forall i.
\end{equation}
The condition $\lambda(A)\leq\bunity$ suggests that there is a strong interplay
between Kleene stars and spectral problems. To describe this in more detail,
we need the following notions and notation.

A cycle $\sigma$ in $\Digr(A)$ is called {\em critical}, if $\mu(\sigma,A)=\lambda(A)$.
Every node and edge that belongs to a critical cycle is called {\em critical}.
The set of
critical nodes is denoted by $\Nca$, the set of critical
edges is denoted by $\Eca$.  The {\em critical digraph} of
$A$, further denoted by $\critgraph(A)=(\Nca,\Eca)$, is the digraph which
consists of all critical nodes and critical edges of $\Digr(A)$. For definite
$A\in\Rnn$, it follows that $a_{ij}a^*_{ji}\leq\bunity$ \cite{BCOQ}. Further,
\begin{equation}
\label{critrule}
(i,j)\in\Eca\Leftrightarrow a_{ij}a^*_{ji}=\bunity.
\end{equation}

For definite $A\in\Rnn$, the relation between Kleene star, critical graph
and spectral problems is briefly as follows \cite{BCOQ,CG:79,SSB}:
\begin{align}
\label{v*adef}
V^*(A)&=
\spann(A^*)=
\left\{\bigoplus_{i=1}^n 
\alpha_i A_{\cdot i}^*,\ 
\alpha_i\in\R
\right\},\\
\label{vadef}
V(A)&=
\left\{\bigoplus_{i\in\Nca} 
\alpha_i A_{\cdot i}^*,\ 
\alpha_i\in\R
\right\},\\
\label{subeigs1}
x\in V^*(A),\; &
(i,j)\in\Eca\Rightarrow 
a_{ij}x_j=x_i.
\end{align}
Equation \eqref{v*adef} means that $V^*(A)$ is the max-algebraic column span of Kleene
star $A^*$, also called {\em Kleene cone}. 
This cone is convex in conventional sense. 
By \eqref{vadef}, $V(A)$ is the max subcone of $V^*(A)$, 
spanned by the columns with
critical indices. Implication \eqref{subeigs1} means that for any subeigenvector 
$x\in V^*(A)$ and
$i\in\Nca$, the maximum in $\bigoplus_j a_{ij} x_j$ is attained at $j$ such that 
$(i,j)\in\Eca$. In
particular, $(A\otimes x)_i=x_i$ for all $x\in V^*(A)$ and $i\in\Nca$.

Not all columns in \eqref{v*adef} and \eqref{vadef} are necessary. 
Let $\critgraph(A)$
have $\critcomps\in\{1,\ldots,n\}$ strongly 
connected components (s.c.c.) 
$\critgraph_{\mu}$, for $\mu=1,\ldots, \critcomps$. 
It follows from the definition of $\critgraph(A)$ that 
s.c.c. $\critgraph_{\mu}$ are disjoint.
The corresponding node sets will be denoted by $N_{\mu}$. 
Let $\noncrit$ denote 
the number of
non-critical nodes of $\Digr(A)$. It can be shown \cite{BCOQ,CG:79} 
that if $i,j$ belong
to the same s.c.c. of $\critgraph(A)$, 
then the columns $A_{\cdot i}^*$ and $A_{\cdot j}^*$
are multiples of each other. The same 
holds for the rows $A_{i\cdot}^*$ and $A_{j\cdot}^*$.
Hence
\begin{align}
\label{v*adef1}
V^*(A)&=\left\{\bigoplus_{i\in K} \alpha_i A_{\cdot i}^*,\ \alpha_i\in\R\right\}\\
\label{vadef2}
V(A)&=\left\{\bigoplus_{i\in\Nca\cap K} \alpha_i A_{\cdot i}^*,\ \alpha_i\in\R\right\},
\end{align}
where $K$ is any set of indices which contains all non-critical indices
and for every $\critgraph_{\mu}$ there is a unique index of
this component in $K$. 

Consider $A_{KK}^*$, 
the principal submatrix of $A^*$ extracted from the rows and columns with indices in $K$. 
Condition \eqref{kleene2} implies that $A_{KK}^*$ is itself a Kleene star. It
follows from the maximality
of $\critgraph_{\mu}$ that there is a unique permutation of $K$ that has 
the greatest weight with respect to
$A_{KK}^*$. The {\em weight of a permutation} $\pi$ of $\{1,\ldots,n\}$ 
with respect to $A\in\Rnn$ is defined as 
$\prod_{i=1}^n a_{i\pi(i)}$. Thus $A_{KK}^*$ is {\em strongly regular} in the sense
of Butkovi\v{c} \cite{But-03}. From this it can be deduced that the columns of $A^*$
with indices in $K$ are {\em independent}, 
meaning that none of them can be expressed as a max
combination of the other columns. In other words \cite{BSS-07}, the columns of 
$A^*$ with indices in $K$ (resp., in $\Nca\cap K$) form a {\em basis} of
$V^*(A)$ (resp., of $V(A)$). This basis is essentially unique 
\cite{BSS-07},
meaning that any other basis can be obtained from it by scalar multiplication.

More precisely, the strong regularity of 
$A_{KK}^*$ is equivalent to saying that this basis is 
{\em tropically independent}, hence the 
tropical rank of $A^*$ is equal to
$\critcomps+\noncrit$, see \cite{AGG,Izh-09,Izh-05} for definitions 
and further details.

\section{Visualization and Boolean matrices}
 
\subsection{Visualization}
Consider a positive $x\in\Rn$ and define
\begin{equation}
X=\diag(x):=
\begin{pmatrix}
x_1 &\ldots &\bzero\\
\vdots & \ddots &\vdots\\
\bzero & \ldots & x_n
\end{pmatrix}
\end{equation}
The transformation $A\mapsto X^{-1}AX$ is called a
{\em diagonal similarity scaling} of $A$.
Such transformations do not change $\lambda(A)$
and $\critgraph(A)$ \cite{ES}. They commute with max-algebraic multiplication
of matrices and hence with the operation of taking the Kleene
star. Geometrically, they correspond to automorphisms of
$\Rn$, both in the case of max algebra and in the case of nonnegative
linear algebra. Further we define scalings which lead to particularly
convenient forms of matrices in max algebra.

A definite matrix $A\in\Rnn$ is called {\em visualized},
if
\begin{align}
\label{visprop1}
&a_{ij}\leq\bunity,\ \forall i,j=1,\ldots, n\\
\label{visprop2}
& a_{ij}=\bunity,\ \forall(i,j)\in\Eca 
\end{align}
A visualized matrix $A\in\Rnn$ is called 
{\em strictly visualized} if
\begin{equation}
\label{strvis}
a_{ij}=\bunity\Leftrightarrow (i,j)\in\Eca.
\end{equation}

Visualization scalings were known already to Afriat \cite{A:63}
and Fiedler-Pt\'{a}k \cite{FP:67}, and motivated extensive study
of matrix scalings in nonnegative linear algebra, see e.g. \cite{ES,ES:75,RSS,SS-91}. 
We remark that some
constructions and facts related to application of visualization
scaling in max algebra 
have been observed in connection with
max algebraic power method \cite{ED-99,ED-01}, 
behaviour of matrix powers \cite{BC-07}
and max-balancing \cite{RSS,SS-91}.

Visualization scalings are described in \cite{SSB} in terms
of the subeigencone $V^*(A)$ and its relative interior.
For the convenience of the reader, we show their
existence for any definite $A\in\Rpnn$.
In the proposition stated below, the summation in part 2. is {\em conventional}.
\begin{proposition}
\label{vis-exist}
Let $A\in\Rpnn$ be definite and $X=\diag(x)$.
\begin{itemize}
\item[1.] If $x=\bigoplus_{i=1}^n A_{\cdot i}^*$
then $X^{-1}AX$ is visualized.
\item[2.] If $x=\sum_{i=1}^n A_{\cdot i}^*$
then $X^{-1}AX$ is strictly visualized.
\end{itemize}
\end{proposition}
\begin{proof} 1. Observe that $x\in V^*(A)$ and $x$ is positive.
Then $a_{ij}x_j\leq x_i$ for all $i,j$ implies
$x_i^{-1}a_{ij}x_j\leq 1$, and  by \eqref{subeigs1}
$x_i^{-1}a_{ij}x_j=1$ for all
$(i,j)\in\Eca$.\\
2. Observe that $x$ is positive,
and that $x\in V^*(A)$ since $V^*(A)$ is convex.
Hence $X^{-1}AX$ is visualized.
It remains to check that $(i,j)\notin\Eca$ implies
$a_{ij}x_j<x_i$. We need to find $k$ such that 
$a_{ij}a_{jk}^*<a_{ik}^*$. But this is true for
$k=i$, since $a_{ii}^*=1$ and $a_{ij}a_{ji}^*<1$ by
\eqref{critrule}. This completes the proof. \end{proof}

More precisely \cite{SSB}, $A\in\Rpnn$ can be visualized by any positive 
vector in $V^*(A)$, and it can be strictly visualized by
any vector in the relative interior of $V^*(A)$.

\subsection{Max algebra and Boolean matrices}
Max algebra is related to the algebra of Boolean
matrices. The latter algebra is defined over the Boolean 
semiring $\bul$ which is the set $\{0,1\}$
equipped with logical operations ``OR'' 
$a\oplus b:=a\vee b$ and ``AND'' $a\otimes b:=a\wedge b$.
Clearly, Boolean matrices can be treated as objects of max
algebra, as a very special but crucial case.

For a strongly connected graph, its cyclicity is defined as the g.c.d.
of the lengths of all cycles (or equivalently, all simple
cycles). If the cyclicity is $1$ then the graph
is called {\em primitive}, otherwise it is called
{\em imprimitive}. We will not distinguish
between cyclicity (or primitivity)
of a Boolean matrix $A$ and the associated digraph $\Digr(A)$. 
Further we recall an important result of Boolean matrix theory.

\begin{proposition}[Brualdi and Ryser \cite{BR}]
\label{brualdi}
Let $A\in\bulnn$ be irreducible, and let $\gamma_A$ be the
cyclicity of $\Digr(A)$ (which is strongly connected). Then
for each $k\geq 1$, there exists a permutation matrix $P$
such that $P^{-1}A^kP$ has $r$ irreducible diagonal
blocks, where $r=\gcd(k,\gamma_A)$, and all elements outside
these blocks are zero. The cyclicity of all these
blocks is $\gamma_A/r$.
\end{proposition}

In max algebra, let $A\in\Rnn$. Define the Boolean matrix
$A^{[C]}=(a_{ij}^{[C]})$ by
\begin{equation}
\label{critmatx}
a_{ij}^{[C]}=
\begin{cases}
1, & (i,j)\in\Eca\\
0, & (i,j)\notin\Eca.
\end{cases}
\end{equation}
Let $A,B\in\Rnn$. Assume that $\critgraph(A)$ has $\critcomps$ s.c.c. 
$\critgraph_{\mu}$ for $\mu=1,\ldots,\critcomps$, with cyclicities $\gamma_{\mu}$.
Denote by $B_{\mu\nu}$ the block of $B$ extracted from the
rows with indices in $N_{\mu}$ and columns with indices in 
$N_{\nu}$.

The following proposition can be seen as a corollary of 
Proposition~\ref{brualdi}. The idea of the proof given below is due to
Hans Schneider. See also \cite{HOW:05} Section 3.1
and \cite{BC-07} Theorem 2.3.

\begin{proposition}
\label{butkovic}
Let $A\in\Rnn$ and $\lambda(A)\neq\bzero$.
\begin{itemize}
\item[1.] $\lambda(A^k)=\lambda^k(A)$.
\item[2.] $(A^{[C]})^k=(A^k)^{[C]}$.
\item[3.] For each $k\geq 1$, there exists a 
permutation matrix $P$ such that 
$(P^{-1}A^kP)_{\mu\mu}^{[C]}$, for each $\mu=1,\ldots,\critcomps,$
has $r_{\mu}:=\gcd(k,\gamma_{\mu})$ irreducible blocks 
and all elements outside
these blocks are zero. The cyclicity of all blocks
in $(P^{-1}A^kP)^{[C]}_{\mu\mu}$ is equal to 
$\gamma_{\mu}/r_{\mu}$.
\end{itemize}
\end{proposition}
\begin{proof} 
We can assume that $A$ is definite.
Further, the diagonal similarity scaling commutes with
max algebraic matrix multiplication and changes
neither $\lambda(A)$ nor $\critgraph(A)$ \cite{ES}, and
by Proposition~\ref{vis-exist}, part 2, there exists a 
strict visualization scaling. Hence we can assume
that $A$ is strictly visualized. In this case $A^{[C]}=A^{[1]}$,
where $A^{[1]}=(a_{ij}^{[1]})$ is defined by
\begin{equation}
a_{ij}^{[1]}=
\begin{cases}
1, & a_{ij}=\bunity,\\
0, & a_{ij}<\bunity.
\end{cases}
\end{equation} 
It is easily seen that $(A^{[1]})^k=(A^k)^{[1]}$. As 
$A^{[1]}=A^{[C]}$, all entries of $A^{[1]}$ outside
the blocks $A_{\mu\mu}^{[1]}$ are zero, which assures
that $(A^{[1]})^k_{\mu\mu}=(A_{\mu\mu}^{[1]})^k$.  

Proposition~\ref{brualdi} implies that part 3. is
true for $(A^{[1]})^k=(A^k)^{[1]}$. This implies that $P^{-1}(A^k)^{[1]}P$
has irreducible blocks and 
$\lambda(A^k)=\bunity$, which shows part 1. Also,
$P^{-1}(A^k)^{[1]}P$ has block structure where all diagonal
blocks are irreducible and all off-diagonal blocks
are zero. This implies $(A^k)^{[C]}=(A^k)^{[1]}$,
and parts 2. and 3. follow immediately.
\end{proof}

\subsection{Cyclic classes}
For a path $P$ in a digraph $G=(N,E)$, where
$N=\{1,\ldots,n\}$, denote by $l(P)$
the length of $P$, i.e., the number of edges traversed by
$P$. 

\begin{proposition}[Brualdi-Ryser \cite{BR}]
\label{ryser}
Let $G=(N,E)$ be a strongly connected digraph with cyclicity $\gamma_G$.
Then the lengths of any two paths connecting $i\in N$ to $j\in N$
(with $i,j$ fixed) are congruent modulo $\gamma_G$.
\end{proposition}

Proposition~\ref{ryser} implies that the 
following {\em equivalence relation} can be defined:
$i\sim j$ if there exists a path $P$ from $i$
to $j$ such that $l(P)\equiv 0(\modd \gamma_G)$. The
equivalence classes of $G$ with respect to this relation
are called {\em cyclic classes} \cite{BV-73, Sem-06, Sem-07}.
The cyclic class of $i$ will be denoted by $[i]$. 

Consider the following {\em access relations} between cyclic
classes: $[i]\to_t[j]$ if there exists a path $P$
from a node in $[i]$ to a node in $[j]$ such that
$l(P)\equiv t(\modd \gamma_G)$. In this case, a path $P$
with $l(P)\equiv t(\modd \gamma_G)$ exists between any
node in $[i]$ and any node in $[j]$. Further, by 
Proposition~\ref{ryser} the length of any path between a node in $[i]$
and a node in $[j]$ is congruent to $t$, so the relation $[i]\to_t [j]$
is well-defined. Classes $[i]$ and $[j]$ will be called {\em adjacent}
if $[i]\to_1 [j]$.

Cyclic classes can be computed in $O(|E|)$ time by
Balcer-Veinott digraph condensation, where $|E|$
denotes the number of edges in $G$. At each step
of this algorithm, we look for all edges which
issue from a certain node $i$, and condense
all end nodes of these edges into a single node. 
A precise description of this method can be found in \cite{BV-73,BR}.
We give an example of its work, see Figures 1 and 2. 
\begin{figure}[h]
\begin{tabular}{ccccccc}
\begin{tikzpicture}[shorten >=1pt,->]
  \tikzstyle{vertex1}=[circle,fill=black!25,minimum size=17pt,inner sep=1pt]
\tikzstyle{vertex2}=[circle,fill=black!40,minimum size=17pt,inner sep=1pt]

\node[vertex2,xshift=0cm,yshift=0cm] (6) at (-150:1.2cm) {$1$};
\foreach \name/\angle/\text in {1/150/2, 2/90/3, 
                                  3/30/4, 4/-30/5, 5/-90/6}
    \node[vertex1,xshift=0cm,yshift=0cm] (\name) at (\angle:1.2cm) {$\text$};
\draw (6) -- (3);
\foreach \from/\to in {1/2,2/3,3/4,4/5,5/6,6/1}
    \draw (\from) -- (\to);
    
\end{tikzpicture}&&
\begin{tikzpicture}[shorten >=1pt,->]
  \tikzstyle{vertex1}=[circle,fill=black!25,minimum size=20pt,inner sep=1pt]
\tikzstyle{vertex2}=[circle,fill=black!40,minimum size=20pt,inner sep=1pt]

\node[vertex1,xshift=-0.71cm,yshift=2.13cm](5){$3$};
\node[vertex2,xshift=0cm,yshift=0cm] (1) at (135:1cm) {$24$};
\foreach \name/\angle/\text in {2/45/5, 
                                  3/-45/6, 4/-135/1}
    \node[vertex1,xshift=0cm,yshift=0cm] (\name) at (\angle:1cm) {$\text$};
\draw (1) .. controls +(110:0.5cm) and +(-110:0.5cm) .. (5);
\draw (5) .. controls +(-70:0.5cm) and +(70:0.5cm) .. (1);
\foreach \from/\to in {1/2,2/3,3/4,4/1}
    \draw (\from) -- (\to);
\end{tikzpicture}&&
\begin{tikzpicture}[shorten >=1pt,->]
  \tikzstyle{vertex1}=[circle,fill=black!25,minimum size=20pt,inner sep=1pt]
\tikzstyle{vertex2}=[circle,fill=black!40,minimum size=20pt,inner sep=1pt]

\node[vertex2,xshift=0cm,yshift=0cm] (2) at (45:1cm) {$35$};

\foreach \name/\angle/\text in {1/135/24,  
                                  3/-45/6, 4/-135/1}
\node[vertex1,xshift=0cm,yshift=0cm] (\name) at (\angle:1cm) {$\text$};
\draw (1) .. controls +(-30:0.5cm) and +(-150:0.5cm) .. (2);
\draw (2) .. controls +(150:0.5cm) and +(30:0.5cm) .. (1);
\foreach \from/\to in {2/3,3/4,4/1}
    \draw (\from) -- (\to);
\end{tikzpicture}&&
\begin{tikzpicture}[shorten >=1pt,->]
  \tikzstyle{vertex1}=[circle,fill=black!25,minimum size=25pt,inner sep=1pt]
\tikzstyle{vertex2}=[circle,fill=black!40,minimum size=25pt,inner sep=1pt]
\node[vertex1,xshift=-1.5cm,yshift=0cm] (1) {$1$};
\node[vertex2,xshift=0cm,yshift=0cm] (2) {$246$};
\node[vertex1,xshift=1.5cm,yshift=0cm] (3) {$35$};
\draw (1) .. controls +(-45:0.7cm) and +(-135:0.7cm) .. (2);
\draw (2) .. controls +(135:0.7cm) and +(45:0.7cm) .. (1);
\draw (2) .. controls +(-45:0.7cm) and +(-135:0.7cm) .. (3);
\draw (3) .. controls +(135:0.7cm) and +(45:0.7cm) .. (2);
\end{tikzpicture}
\end{tabular}
\caption{Balcer-Veinott algorithm\label{f:bvalg1}}
\end{figure}
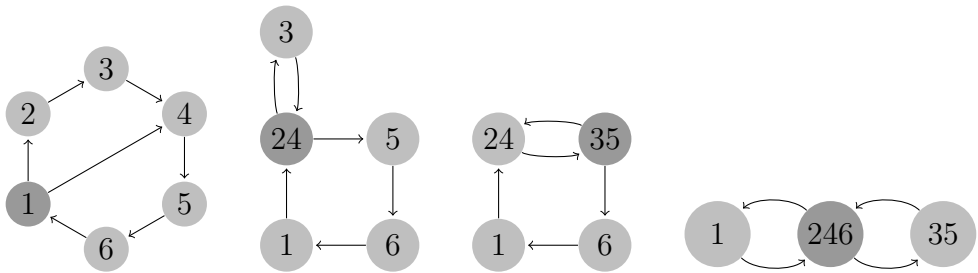

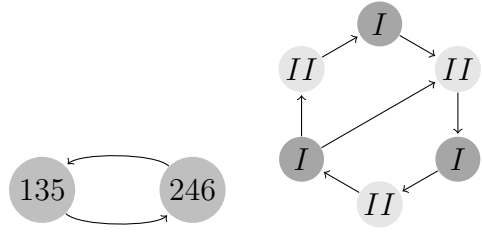
\begin{figure}[h]
\begin{tabular}{ccc}
\begin{tikzpicture}[shorten >=1pt,->]
  \tikzstyle{vertex}=[circle,fill=black!25,minimum size=25pt,inner sep=1pt]
\node[vertex,xshift=-1cm,yshift=0cm] (1) {$135$};
\node[vertex,xshift=1cm,yshift=0cm] (2) {$246$};
\draw (1) .. controls +(-45:0.7cm) and +(-135:0.7cm) .. (2);
\draw (2) .. controls +(135:0.7cm) and +(45:0.7cm) .. (1);
\end{tikzpicture}
&&
\begin{tikzpicture}[shorten >=1pt,->]
  \tikzstyle{vertex1}=[circle,fill=black!10,minimum size=17pt,inner sep=1pt]
\tikzstyle{vertex2}=[circle,fill=black!35,minimum size=17pt,inner sep=1pt]
\foreach \name/\angle/\text in {1/150/II, 
                                  3/30/II,  5/-90/II}
    \node[vertex1,xshift=0cm,yshift=0cm] (\name) at (\angle:1.2cm) {$\text$};
\foreach \name/\angle/\text in { 2/90/I, 
                                   4/-30/I,  6/-150/I}
    \node[vertex2,xshift=0cm,yshift=0cm] (\name) at (\angle:1.2cm) {$\text$};

\draw (6) -- (3);
\foreach \from/\to in {1/2,2/3,3/4,4/5,5/6,6/1}
    \draw (\from) -- (\to);
\end{tikzpicture}
\end{tabular} 
\caption{Result of the algorithm (left) and cyclic classes (right)\label{f:bvalg2}}
\end{figure}

In this example, see Figure 1 at the left,
we start by condensing nodes $2$ and $4$, which are
``next to'' node $1$, into the node $24$. Further we proceed
with condensing nodes $3$ and $5$ into the node $35$. 
In the end, see Figure 2 at the left, there are just two nodes
$135$ and $246$. They correspond to two cyclic classes $\{1,3,5\}$
and $\{2,4,6\}$ of the initial graph, see Figure 2 at the right. 
 
The notion of cyclic classes and access relations 
can be generalized to the
case when $G$ has $\critcomps$ disjoint components $G_{\mu}$
with cyclicities $\gamma_{\mu}$, for $\mu=1,\ldots,\critcomps$ (just like the critical graph
in max algebra).
In this case we write $i\sim j$ if 
$i,j$ belong to the same component and there exists a
path $P$ from $i$ to $j$ such that $l(P)\equiv 0(\modd \gamma_{\mu})$.
If $l(P)\equiv t(\modd \gamma_{\mu})$, then we write $[i]\to_t[j]$.
In this case the {\em cyclicity} of $G$ is 
$\gamma:=\lcm\ \gamma_{\mu},\ \mu=1,\ldots,\critcomps.$

We will be interested in the cyclic classes of critical graphs,
and below we also give an explanation of these, in
terms of the Boolean matrix $A^{[C]}$. Let $A\in\Rnn$.
Following Brualdi and Ryser \cite{BR} we can find such 
ordering of the indices that
any submatrix $A^{[C]}_{\mu\mu}$, which corresponds 
to an imprimitive
component $\critgraph_{\mu}$ of $\critgraph(A)$, will be of the form
\begin{equation}
\label{crit-subm}
\begin{pmatrix}
\blzero & A^{[C]}_{s_1s_2} & \blzero &\cdots & \blzero\\
\blzero & \blzero & A^{[C]}_{s_2s_3} & \cdots & \blzero\\
\vdots & \vdots & \vdots &\ddots &\vdots\\
\blzero & \blzero & \blzero & \cdots & A^{[C]}_{s_{k-1}s_k}\\
A^{[C]}_{s_ks_1} & \blzero &\blzero & \cdots & \blzero
\end{pmatrix},
\end{equation}
where $k$ is the number of cyclic classes in $\critgraph_{\mu}$. 
Indices $s_i$ and $s_{i+1}$ for $i=1,\ldots,k-1$, and 
$s_k$ and $s_1$ correspond to adjacent cyclic classes.
By Proposition \ref{butkovic} part 2, when $A$ is raised
to power $k$, $A^{[C]}$ is also raised to the same power
over the Boolean algebra.
Any power of $A^{[C]}$ has a similar block-permutation
form. In particular, $(A^{\gamma_{\mu}})^{[C]}_{\mu\mu}$ looks like
\begin{equation}
\label{crit-ak}
\begin{pmatrix}
(A^{\gamma_{\mu}})^{[C]}_{s_1s_1} & \blzero & \blzero &\cdots & \blzero\\
\blzero & (A^{\gamma_{\mu}})^{[C]}_{s_2s_2} &\blzero & \cdots & \blzero\\
\vdots & \vdots & \vdots &\ddots &\vdots\\
\blzero & \blzero & \blzero & \cdots & (A^{\gamma_{\mu}})^{[C]}_{s_ks_k}
\end{pmatrix}
\end{equation}

Theorem 5.4.11 of \cite{Kim:82} implies that the sequence
$(A^k)^{[C]}=(A^{[C]})^k$ becomes periodic after $k\leq (n-1)^2+1$, with period
$\caa=\lcm(\gamma_{\mu}),\ \mu=1,\ldots,\critcomps$.
In the periodic regime, all entries of nonzero blocks 
are equal to $1$. 

\section{Periodicity and complexity}

\subsection{Spectral projector and periodicity in max algebra}
We further assume that
the critical graph $\crit(A)$ occupies the first $c$ nodes, i.e.,
that $\Nca=\{1,\ldots,\critnodes\}$.

For a definite $A\in\Rnn$,
consider the matrix $Q(A)$ with entries
\begin{equation}
\label{e:specproj}
q_{ij}=\bigoplus_{k=1}^{\critnodes} a^*_{ik}a^*_{kj},\ i,j=1,\ldots,n.
\end{equation}
The max-linear operator whose matrix is $Q(A)$,
is a max-linear {\em spectral projector}
associated with $A$, in the sense that it projects $\R^n$
on the max-algebraic eigencone $\{x\mid A\otimes x=x\}$
\cite{BCOQ}. We also note that $Q(A)$ is important for the
policy iteration algorithm of \cite{CGG-99}.

We will need the following property of $Q(A)$ which follows
directly from \eqref{e:specproj}.

\begin{proposition}
\label{critkls}
For a definite $A\in\Rnn$, any column (or row)
of $Q(A)$ with index in $1,\ldots,\critnodes$
is equal to the corresponding column (or row) of $A^*$.
\end{proposition}

This operator is closely related
to the periodicity questions.
\begin{theorem}[Baccelli et al. \cite{BCOQ}, Theorem 3.109]
\label{baccelli}
Let $A\in\Rnn$ be irreducible and definite,
and let all s.c.c. of $\crit(A)$ have cyclicity $1$.
Then there is an integer $T(A)$
such that $A^r=Q(A)$ for all $r\geq T(A)$.
\end{theorem}

It can be easily shown that Theorem~\ref{baccelli}
can be generalized to the situation when $A$ is in a
blockdiagonal form with irreducible definite
diagonal blocks $A_{11},\ldots,A_{uu}$ so that
\begin{equation}
\label{def:direct-sum}
A=
\begin{pmatrix}
A_{11} &\dots & \blzero\\
\vdots &\ddots &\vdots\\
\blzero &\dots & A_{uu}
\end{pmatrix}
\end{equation}

If $A$ is irreducible and definite, then $A^k$ for all
$k\geq 1$ is in a blockdiagonal form \eqref{def:direct-sum},
where all blocks $A_{11},\ldots,A_{uu}$ are irreducible
and definite.

Indeed, it is evident (when $A$ is assumed
to be visualised), that the m.c.g.m. of each block does not exceed
$1$.
On the other hand,
any column $A^*_{\cdot i}$ with $1\leq i\leq \critnodes$ is a positive
max-algebraic eigenvector
of $A$ and hence of $A^{\gamma}$. This shows that the m.c.g.m.
of each submatrix, being equal to the largest max-algebraic eigenvalue
of that submatrix, cannot be less than $1$.

It follows from Proposition \ref{butkovic} part 3 that
all components of
$\crit(A^{\caa})$ are primitive, where
$\caa$ is the cyclicity of $\critgraph(A)$.

These arguments lead us to the following extension
of Theorem \ref{baccelli}.

\begin{theorem}
\label{cohen}
Let $A\in\Rnn$ be irreducible and definite, and let
$\caa$ be the cyclicity of $\crit(A)$.
There exists $T(A)$ such that
\begin{itemize}
\item[1.] $A^{t\caa}=Q(A^{\caa})$ for all $t\caa\geq T(A)$;
\item[2.] $A^{r+\caa}=A^r$ for all $r\geq T(A)$.
\end{itemize}
\end{theorem}

$T(A)$ will be called the {\em transient},
and the powers $A^r$ for $r\geq T(A)$ will be called
{\em periodic powers}.

It is also important that the entries $a_{ij}^{(r)}$, where
$i$ or $j$ are critical, become periodic much faster than
the non-critical part of $A$. The following proposition
is a known result, which
is proved here for convenience of the reader. 

\begin{proposition}[Nachtigall \cite{Nacht}]
\label{p:nacht}
Let $A\in\Rnn$ be a definite irreducible matrix. 
Critical rows and columns of $A^r$ become
periodic for $r\geq n^2$. 
\end{proposition} 
\begin{proof} 
We prove the claim for rows, and for columns
everything is analogous.
Let $i\in \{1,\ldots,\critnodes\}$. Then there is a critical
cycle of length $l$ to which $i$ belongs. Hence
$a_{ii}^{(kl)}=\bunity$ for $k\geq 1$. Since for all
$m<k$ and any $t=1,\ldots,n$ we have 
$$a_{is}^{(ml)}=a_{ii}^{((k-m)l)}a_{is}^{(ml)}\leq a_{is}^{(kl)},$$  
it follows that 
\begin{equation}
\label{entries1}
a_{is}^{(kl)}=\bigoplus_{m=1}^k a_{is}^{(ml)}.
\end{equation}
Entries $a_{is}^{(kl)}$ are maximal weights of paths of length
$k$ with respect to the matrix $A^{l}$. Since the weights 
of all cycles are less than or equal to $\bunity$ and all
paths of length $n$ are not simple, 
the maximum is achieved at $k\leq n$. Using \eqref{entries1}
we obtain that $a_{is}^{((t+1)l)}=a_{is}^{(tl)}$ for
all $t\geq n$. Further,
$$a_{is}^{(tl+d)}=\bigoplus_k a_{ik}^{(tl)}a_{ks}^{(d)},$$
and it follows that $a_{is}^{((t+1)l+d)}=a_{is}^{(tl+d)}$
for all $t\geq n$ and $0\leq d\leq l-1$. Hence
$a_{is}^{(k)}$ is periodic for $k\geq nl$, and
all these sequences, for any $i=1,\ldots,\critnodes$
and any $s$, become periodic for $k\geq n^2$. 
\end{proof} 
The number after which the critical rows and columns
of $A^t$ become periodic will be denoted by $T_c(A)$.

\subsection{The ultimate spans of matrices}
Max algebraic powers in the periodic regime 
have the following properties.

\begin{proposition}
\label{p:lindep}
Let $A\in\Rnn$ be a definite and irreducible matrix, and
let $t\geq 0$ be such that $t\gamma\geq T(A)$.
Then for every integer $l\geq 0$
\begin{equation}
\label{e:lindep}
A_{k\cdot}^{t\gamma+l}=
\bigoplus_{i=1}^c a_{ki}^{(t\gamma)} A_{i\cdot}^{t\gamma+l},\ 
A_{\cdot k}^{t\gamma+l}=\bigoplus_{i=1}^c a_{ik}^{(t\caa)} 
A_{\cdot i}^{t\caa+l},\quad 1\leq k\leq n.
\end{equation}
\end{proposition}
\begin{proof}
Due to Theorem \ref{cohen},
for $B=A^{\caa}$ and any $r\geq T(B)$ we have
\begin{equation}
\label{bij}
b_{kj}^{(r)}=\bigoplus_{i=1}^c b_{ki}^*b_{ij}^*,\quad 1\leq k,j\leq n.
\end{equation}
By Theorem \ref{cohen} and Proposition \ref{critkls}, 
we have $b_{ki}^*=b_{ki}^{(r)}=a_{ki}^{(t\caa)}$ and $b_{ij}^*=b_{ij}^{(r)}=a_{ij}^{(t\caa)}$
for all $r\geq T(B)$ or equivalently $t\caa\geq T(A)$, and any $i\leq c$. Hence
\begin{equation}
\label{aijr}
a_{kj}^{(t\gamma)}=\bigoplus_{i=1}^c a_{ki}^{(t\gamma)}a_{ij}^{(t\gamma)},\quad 1\leq k,j\leq n.
\end{equation}
In the matrix notation, this is equivalent to:
\begin{equation}
\label{e:lindepprel}
A_{k\cdot}^{t\gamma}=
\bigoplus_{i=1}^c a_{ki}^{(t\gamma)} A_{i\cdot}^{t\gamma},\ 
A_{\cdot k}^{t\gamma}=\bigoplus_{i=1}^c a_{ik}^{(t\gamma)} 
A_{\cdot i}^{t\gamma},\quad 1\leq k\leq n.
\end{equation}
Multiplying \eqref{e:lindepprel} by any power $A^l$, we
obtain \eqref{e:lindep}.
\end{proof}

In the proof of the next proposition we will use the
following simple principle
\begin{equation}
\label{bell-opt}
a_{ij}^{(r)}a_{jk}^{(s)}\leq a_{ik}^{(r+s)},\quad\forall i,j,k,r,s,
\end{equation}
which holds for the matrix powers in max algebra.

\begin{proposition}
\label{p:bellman2}
Let $A\in\Rnn$ be a definite and irreducible matrix, and
let $i,j\in \{1,\ldots,c\}$ be such that $[i]\to_l[j]$, 
for some $0\leq l<\caa$.
\begin{itemize}
\item[1.] For any $r\geq T_c(A)$, there exists $t_1\geq 0$ 
such that
\begin{equation}
\label{e:bellman2}
a_{ij}^{(t_1\caa+l)} A_{\cdot i}^r=A_{\cdot j}^{(r+l)},\ 
a_{ij}^{(t_1\caa+l)} A_{j\cdot}^r=A_{i\cdot}^{r+l}.
\end{equation}
\item[2.] If $A$ is visualized, then for all $r\geq T_c(A)$
\begin{equation}
\label{e:bellman3}
A_{\cdot i}^r=A_{\cdot j}^{r+l},\ 
A_{j\cdot}^r=A_{i\cdot}^{r+l}.
\end{equation}
\end{itemize}
\end{proposition}
\begin{proof}
If $[i]\to_l[j]$ then $[j]\to_s [i]$ where $l+s=\caa$.
By the definition of access relations there exists a critical path
of length $t_1\caa+l$ connecting $i$ to $j$, and a critical path
of length $t_2\caa+s$ connecting $j$ to $i$. Hence
$a_{ij}^{(t_1\caa+l)}a_{ji}^{(t_2\caa+s)}=1$, and in the visualized
case $a_{ij}^{(t_1\caa+l)}=a_{ji}^{(t_2\caa+s)}=1$. Combining this
with \eqref{bell-opt} we obtain
\begin{equation}
\label{bmn-ineqs}
\begin{split}
A_{\cdot i}^r&=A_{\cdot i}^r a_{ij}^{(t_1\caa+l)}a_{ji}^{(t_2\caa+s)}\leq
A_{\cdot j}^{r+t_1\caa+l}a_{ji}^{(t_2\caa+s)}\leq A_{\cdot i}^{r+(t_1+t_2+1)\caa},\\
A_{j\cdot}^r&=A_{j\cdot}^r a_{ij}^{(t_1\caa+l)}a_{ji}^{(t_2\caa+s)}\leq
A_{i\cdot}^{r+t_1\caa+l}a_{ji}^{(t_2\caa+s)}\leq A_{j\cdot}^{r+(t_1+t_2+1)\caa}.
\end{split}
\end{equation}
Since $r\geq n^2$, by Proposition \ref{p:nacht}
$A_{\cdot i}^r=A_{\cdot i}^{r+(t_1+t_2+1)\caa}$ and 
$A_{j\cdot}^r=A_{j\cdot}^{r+(t_1+t_2+1)\caa}$, hence all inequalities
\eqref{bmn-ineqs} are equalities. Multiplying them by $a_{ij}^{(t_1\caa+l)}$
we obtain \eqref{e:bellman2}, which is \eqref{e:bellman3} in the 
visualized case. 
\end{proof}

Proposition \ref{p:bellman2} says that in any power
$A^r$ for $r\geq n^2$, the critical columns (or rows) can be obtained from
the critical columns (or rows) of the spectral projector
$Q(A^{\caa})$ via a permutation whose cycles are determined
by the cyclic classes of $\critgraph(A)$. Proposition \ref{p:lindep}
adds to this that 
all non-critical columns (or rows) of any periodic 
power are in the max cone spanned by the critical columns
(or rows). From this we conclude the following.

\begin{proposition}
\label{p:ultspan}
All powers $A^r$ for $r\geq T(A)$ have the same column span,
which is the eigencone $V(A^{\caa})$.
\end{proposition}

Proposition \ref{p:ultspan} enables us to say that 
$V(A^{\caa})$ is the {\em ultimate column span} of
$A$. Similarly, we have the {\em ultimate row span} which is
$V((A^T)^{\caa})$. These cones are generated by 
critical columns (or rows)
of the Kleene star $(A^{\caa})^*$. For a basis of this cone,
we can take any set of columns $(A^{\caa})^*$
(equivalently $Q(A^{\caa})$ or
$A^r$ for $r\geq T(A)$), 
whose indices form a minimal set of representatives
of all cyclic classes of $\critgraph(A)$. This basis is tropically independent
in the sense of \cite{AGG, Izh-05,Izh-09}.

\subsection{Solving periodicity problems by square 
multiplication}
Let $A\in\Rnn$ and $\lambda(A)=\bunity$.
The {\em $t$-attraction cone} $\operatorname{Attr}(A,t)$ is the max cone which consists of
all vectors $x$, for which there exists an integer $r$ such that 
$A^r\otimes x=A^{r+t}\otimes x$, and hence
this is also true for all integers greater than or equal to $r$.
Actually we may speak of any
$r\geq T(A)$, due to the following observation.

\begin{proposition}
\label{allr}
Let $A$ be irreducible and definite.
The systems $A^r\otimes x=A^{r+t}\otimes x$ are equivalent for all $r\geq T(A)$.
\end{proposition}
\begin{proof}
Let $x$ satisfy $A^s\otimes x=A^{s+t}\otimes x$ for some 
$s\geq T(A)$, then it also satisfies
this system for all greater $s$. 
Due to the periodicity, for all $k$ from $T(A)\leq k\leq s$ there
exists $l>s$ such that $A^k=A^l$. Hence $A^k\otimes x=A^{k+t}\otimes x$ also hold for
$T(A)\leq k\leq s$. 
\end{proof}

\begin{corollary}
\label{attrt-attr1}
$\operatorname{Attr}(A,t)=\operatorname{Attr}(A^t,1).$
\end{corollary}
\begin{proof}
By Proposition \ref{allr}, $\Attr(A,t)$ is solution set to the system
$A^r\otimes x=A^{r+t}\otimes x$ for any $r\geq T(A)$ which is a multiple of $t$,
which proves the statement.
\end{proof}

An equation 
of $A^r\otimes x=A^{r+t}\otimes x$ with index in 
$\{1,\ldots,\critnodes\}$ 
will be called {\em critical},
and the subsystem of the first $c$ equations 
will be called the {\em critical subsystem}.

\begin{proposition}
\label{onlycrit}
Let $A$ be irreducible and definite and let $r\geq T(A)$. Then
$A^r\otimes x=A^{r+t}\otimes x$ is equivalent to its critical subsystem.
\end{proposition}
\begin{proof}
Consider a non-critical component $A^r_{k\cdot}\otimes x=A^{r+t}_{k\cdot}\otimes x$. 
Using \eqref{e:lindep} it can be written as
\begin{equation}
\bigoplus_{i=1}^{\critnodes} 
a_{ki}^{(r)} A^r_{i\cdot}\otimes x=
\bigoplus_{i=1}^{\critnodes} a_{ki}^{(r)} A^{r+t}_{i\cdot}\otimes x,
\end{equation}
hence it is a max combination of equations in the critical subsystem.
\end{proof}

Next we give a bound on the computational complexity of deciding whether
$x\in \Attr(A,t)$, as well as other related problems which we formulate below.

\msn
P1. For a given $x$, decide whether $x\in\Attr(A,t)$.\\
P2. For a given $k:\; 0\leq k<\caa$, compute 
periodic power $A^r$ where $r\equiv k(\modd\;\caa)$.\\
P3. For a given $x$ compute the ultimate period of $\{A^r\otimes x,\; r\geq 0\}$, meaning the
least integer $\alpha$ such that $A^{r+\alpha}\otimes x=A^r\otimes x$
for all $r\geq T(A)$.

The following theorem is analogous
to the results of Seman\v{c}\'{\i}kov\'a \cite{Sem-06,Sem-07}.

\begin{theorem}
\label{blanka}
For any irreducible matrix $A\in\Rnn$, the problems P1-P3
can be solved in $O(n^3\log n)$ time.
\end{theorem}
\begin{proof}
First note that we can compute both $\lambda(A)$ 
and a subeigenvector, and identify all critical
nodes in no more than $O(n^3)$ operations, 
which is done essentially by Karp
and Floyd-Warshall algorithms \cite{PS}. Further we can 
identify all cyclic classes of $\critgraph(A)$
by Balcer-Veinott condensation in $O(n^2)$ operations. 

By Proposition \ref{p:nacht} the critical rows and columns 
become periodic for 
$r\geq n^2$. To know the critical rows and columns
of a given power $r'\geq T(A)$, it suffices to compute $A^r$ 
for arbitrary $r\geq n^2$
which can be done in $O(\log n)$ 
matrix squaring ($A$, $A^2$, $A^4$, ...) and takes
$O(n^3\log n)$ time, and to apply the corresponding
permutation on cyclic classes which takes $O(n^2)$ overrides. 
By Proposition \ref{onlycrit}
we readily solve P1 by the verification 
of the critical subsystem of $A^{r'}\otimes x=A^{r'+t}\otimes x$ which takes
$O(n^2)$ operations. Using linear 
dependence \eqref{e:lindep} the remaining non-critical submatrix
of $A^{r}$, for any $r\geq T(A)$
such that $r\equiv k(\modd\;\gamma)$, can be computed in $O(n^3)$ time.
This solves P2.

As the non-critical rows of $A$ are generated by the critical
rows, the ultimate period of $\{A^r\otimes x\}$ is determined by the
critical components. For visualized matrix we know that
$A_{i\cdot}^{r+t}=A_{j\cdot}^r$ for all $i,j$ such that
$[i]\to_t[j]$. This implies $(A^{r+t}\otimes x)_i=(A^r\otimes x)_j$ for
$[i]\to_t[j]$, meaning that, to determine the period we need
only the critical subvector of $A^r\otimes x$ for any fixed
$r\geq n^2$. Indeed, for any $i\in N_\critgraph(A)$ and $r\geq n^2$ the sequence
$\{(A^{r+t}\otimes x)_i,\;t\geq 0\}$ can be represented as a sequence of
critical coordinates of $A^r\otimes x$ determined by a permutation
on $\gamma_{\mu}$ cyclic classes of the s.c.c. to which $i$
belongs. To compute the period, we take a sample of $\gamma_{\mu}$
numbers appearing consecutively in the sequence, and check all possible periods,
which takes no more than $\gamma_{\mu}^2$ operations.  
The period of $A^r\otimes x$ appears
as the l.c.m. of these periods. It remains to note that
all operations above do not require more than $O(n^3)$
time. This solves P3.
\end{proof}

\section{Properties of periodic powers}

\subsection{Core matrix}
In the sequel we always assume that $A\in\Rnn$
is irreducible.
Let $\crit(A)$
consist of $\critcomps$ s.c.c. $\crit_{\mu}$ with cyclicities
$\gamma_{\mu}$, for $\mu=1,\ldots,\critcomps$. Let
$\caa=$ l.c.m.$(\gamma_{\mu})$ and $\noncrit$
be the number of non-critical nodes. Further it will be convenient
(though artificial) to consider, together with
these components, also
``non-critical components'' $\crit_{\mu}$ for
$\mu=\critcomps+1,\ldots,\critcomps+\noncrit$,
whose
node sets $N_{\mu}$ consist of just one
non-critical node, and the set of edges
is empty.

Consider the block decomposition of $A^r$ for $r\geq 1$,
induced by the subsets $N_{\mu}$ for
$\mu=1,\ldots,\critcomps+\noncrit$.
The submatrix of $A^{r}$ extracted form the rows in
$N_{\mu}$ and columns in $N_{\nu}$ will be denoted by
$A_{\mu\nu}^{(r)}$. If $A$ is visualized and definite,
we define the corresponding {\em core matrix}
$A^{\Core}=(\alpha_{\mu\nu}),\
\mu,\nu=1,\ldots,\critcomps+\noncrit$ by
\begin{equation}
\alpha_{\mu\nu}=\max\{a_{ij}\mid i\in N_{\mu},\, j\in
N_{\nu}\}.
\end{equation}
The entries of $(A^{\Core})^*$ will be denoted by $\alpha^*_{\mu\nu}$.
Their role is
investigated in the next theorem.

\begin{theorem}
\label{t:maxentries}
Let $A\in\Rnn$ be a definite visualized matrix and $r\geq T_c(A)$. Let
$\mu,\nu=1,\ldots,\critcomps+\noncrit$ be
such that at least one of these indices is critical.
Then the maximal entry of the block
$A^{(r)}_{\mu\nu}$
is equal to $\alpha^*_{\mu\nu}$.
\end{theorem}
\begin{proof}
The entry $\alpha^*_{\mu\nu}$ is the maximal weight over paths from $\mu$ to $\nu$, with respect to the matrix
$A^{\Core}$. We take such a
path $(\mu_1,\ldots,\mu_l)$ with maximal weight,
where $\mu_1:=\mu$ and $\mu_l=\nu$.
With this path we can associate a path
$\pi$ in $\Digr(A)$ defined by
$\pi=\tau_1\circ\sigma_1\circ
\tau_2\circ\ldots\circ\sigma_{l-1}\circ\tau_l$,
where $\tau_i$ are critical paths which
entirely belong to the components $\crit_{\mu_i}$, and
$\sigma_i$ are edges with maximal weight connecting $\crit_{\mu_i}$
to $\crit_{\mu_{i+1}}$. Such a path exists
since any two nodes in the same component
$\crit_{\mu}$ can be connected to each other by critical paths
if $\mu$ is critical, and if $\mu$ is non-critical then
$\crit_{\mu}$ consists just of one node.
The weights of $\tau_i$ are equal to $\bunity$, hence the weight of $\pi$ is equal to $\alpha^*_{\mu\nu}$.
It follows from the definition of
$\alpha_{\mu\nu}$ and $\alpha^*_{\mu\nu}$ that $\alpha_{\mu\nu}^*$
is the greatest weight
over all paths which connect nodes in
$\crit_{\mu}$ to nodes in $\crit_{\nu}$. As at
least one of the indices $\mu,\nu$
is critical, there is freedom in the choice of
the paths $\tau_1$ or $\tau_l$ which can be of arbitrary length.
Assume w.l.o.g. that $\mu$ is critical. Then
for any $r$ exceeding the length of $\sigma_1\circ
\tau_2\circ\ldots\circ\sigma_{l-1}\circ\tau_l$ which we denote by
$l_{\mu\nu}$, the block
$A^{(r)}_{\mu\nu}$ contains an entry equal to
$\alpha^*_{\mu\nu}$ which is the greatest entry of the block.
Taking the maximum $T'(A)$ of $l_{\mu\nu}$ over all ordered pairs
$(\mu,\nu)$ with $\mu$ or $\nu$ critical, we obtain the claim
for $r\geq T'(A)$. Evidently, $T'(A)$ can be replaced by
$T_c(A)$.
\end{proof}

\subsection{CSR-representation}

For a definite visualized matrix $A\in\Rnn$,
the statements of Propositions
\ref{p:lindep} and \ref{p:bellman2} can be combined in the
following. Let $C\in\R^{n\times\critnodes}$ and
$R\in\R^{\critnodes\times n}$
be matrices extracted from the first
$\critnodes$ columns
(resp. rows) of $Q(A^{\caa})$
(or equivalently $(A^{\caa})^*$),
and let $S:=A^{[C]}$, the critical
matrix of $A$ defined by \eqref{critmatx}.

\begin{theorem}[Schneider \cite{HS:PC}]
\label{schneider}
Let $A\in\Rnn$ be definite and visualized. For $r\geq T(A)$
and $r\equiv l(\modd\; \caa)$, $A^r=C\otimes S^l\otimes R$.
\end{theorem}
\begin{proof}
By \eqref{e:specproj}
$Q(A^{\gamma})=C\otimes R$, and by Theorem \ref{cohen}
$A^{\caa t}=Q(A^{\gamma})=C\otimes R$ for $\gamma t\geq T(A)$.
Thus the claim is already proved for $r=\caa t\geq T(A)$.

Note that $C$ (resp. $R$) can be extracted from the first
$\critnodes$ columns (resp. rows) of $A^{\gamma t}$.

As $S=A^{[C]}$, it follows that 1) the $(i,j)$ entry of
$S^l$ can be $1$ only if $[i]\to_l[j]$,
2) for each pair of classes $[i]\to_l[j]$ and
each $i_1\in[i]$ there exists $j_1\in [j]$ such that
the $(i_1,j_1)$ entry of $S^l$ equals $1$.

Using these
two observations and equation \eqref{e:bellman2} applied
to $A^{(\caa t)}$ and $A^{(\caa t+l)}$ for $\caa t\geq T(A)$,
we obtain that the critical columns of $A^{(\caa t+l)}$
are given by $C\otimes S^l$ and the critical rows of
$A^{(\caa t+l)}$ are given by $S^l\otimes R$.

Combining this with any of the two equations of
\eqref{e:lindep}, we obtain that $A^{\caa t+l}=C\otimes S^l\otimes R$,
for any $\caa t\geq T(A)$.
\end{proof}

Further we observe that the dimensions of
periodic powers and the {\em $CSR$-representation}
established in Theorem \ref{schneider} can be reduced.

The rows and columns with
indices in the same cyclic class coincide
in any power $A^r$, where $r\geq T_c(A)$ and
$A$ is definite and
visualized. Hence
the blocks of $A^r$, for $\mu,\nu=1,\ldots,\critcomps+\noncrit$
and $r\geq T_c(A)$, are of the form
\begin{equation}
\label{munublocks}
A_{\mu\nu}^{(r)}=
\begin{pmatrix}
\Tilde{a}^{(r)}_{s_1t_1}E_{s_1t_1} &\dots &\Tilde{a}^{(r)}_{s_1t_m} E_{s_1t_m}\\
\vdots &\ddots &\vdots\\
\Tilde{a}^{(r)}_{s_kt_1}E_{s_kt_1} &\dots &\Tilde{a}^{(r)}_{s_kt_m} E_{s_kt_m}
\end{pmatrix},
\end{equation}
where $k$ (resp. $m$) are cyclicities of $\crit_{\mu}$
(resp. $\crit_{\nu}$), indices $s_1,\ldots,s_k$
(resp. $t_1,\ldots t_m$) correspond to properly
arranged cyclic classes of $\crit_{\mu}$ (resp. $\crit_{\nu}$),
and $E_{s_it_j}$ are matrices with appropriate dimensions
with all entries equal to $1$. We assume that $\crit_{\mu}$
has just one ``cyclic class'' if $\mu$ is non-critical.

Formula \eqref{munublocks}
defines the matrix $\Tilde{A}^{(r)}\in
\R^{(\cyclasses+\noncrit)\times(\cyclasses+\noncrit)}$,
where $\cyclasses$ is the total number of cyclic classes,
as matrix with entries
$\Tilde{a}^{(r)}_{s_it_j}$. Corresponding to \eqref{munublocks},
this matrix has blocks
\begin{equation}
\label{munublocks-tilde}
\Tilde{A}_{\mu\nu}^{(r)}=
\begin{pmatrix}
\Tilde{a}^{(r)}_{s_1t_1} &\dots &\Tilde{a}^{(r)}_{s_1t_m}\\
\vdots &\ddots &\vdots\\
\Tilde{a}^{(r)}_{s_kt_1} &\dots &\Tilde{a}^{(r)}_{s_kt_m}
\end{pmatrix}.
\end{equation}

It follows that $\Tilde{A}^{(r_1+r_2)}=\Tilde{A}^{(r_1)}\otimes
\Tilde{A}^{(r_2)}$ for all $r_1,r_2\geq T_c(A)$.
In words, the multiplication of any two powers $A^{(r_1)}$
and $A^{(r_2)}$ for $r_1,r_2\geq T_c(A)$ reduces to the
multiplication of $\Tilde{A}^{(r_1)}$ and $\Tilde{A}^{(r_2)}$.

If we take $r=\caa t+l\geq T(A)$ (instead of $T_c(A)$ above)
and denote $\Tilde{A}:=\Tilde{A}^{(\caa t+1)}$, then
due to the periodicity we obtain
\begin{equation}
\label{e:forta}
\Tilde{A}^{(\caa t+l)}=\Tilde{A}^{((\caa t+1)l)}=\Tilde{A}^l=
\Tilde{A}^{\caa t+l},
\end{equation}
so that $\Tilde{A}^{(r)}$ can be regarded as the $r$th power
of $\Tilde{A}$, for all $r\geq T(A)$.

Matrices $C$, $R$, and $S^t$ for $t\geq (n-1)^2+1$ have the same
block structure as in \eqref{munublocks}.
This shows that the behavior
of periodic powers $A^r$ is fully described by
\begin{equation}
\label{CSR-red}
\Tilde{A}^r=\Tilde{C}\Tilde{S}^l\Tilde{R},\quad r\equiv l(\modd\;\caa),
\end{equation}
where $\Tilde{S}$ is a $\cyclasses\times\cyclasses$
Boolean matrix with blocks
\begin{equation}
\label{tildes}
\Tilde{S}_{\mu\mu}=
\begin{pmatrix}
0 & 1 & 0 &\cdots & 0\\
0 & 0 & 1 & \cdots & 0\\
\vdots & \vdots & \vdots &\ddots &\vdots\\
0 & 0 & 0 & \cdots & 1\\
1 & 0 & 0 & \cdots & 0
\end{pmatrix},\quad
\Tilde{S}_{\mu\nu}=0 \ \text{for $\mu\neq\nu$,}
\end{equation}
and $\Tilde{C}\in\R^{(\cyclasses+\noncrit)\times\cyclasses}$ and
$\Tilde{R}\in\R^{\cyclasses\times(\cyclasses+\noncrit)}$
are extracted from critical
columns and rows of $\Tilde{A}^{\caa}=Q(\Tilde{A}^{\caa})$, or equivalently,
formed from the scalars in the blocks of $C$ and $R$.

\subsection{Circulant properties}

\label{ss:critcirc}

Matrix $A\in\Rnn$ is called a {\em circulant} if there exist scalars
$\alpha_1,\ldots,\alpha_n$ such that $a_{ij}=\alpha_d$ whenever $j-i=d(\modd\; n)$.
This looks like
\begin{equation}
\label{def-circulant}
A=
\begin{pmatrix}
\alpha_1 & \alpha_2 & \alpha_3 & \cdots & \alpha_n\\
\alpha_n & \alpha_1 & \alpha_2 & \cdots & \alpha_{n-1}\\
\alpha_{n-1} & \alpha_n & \alpha_1 & \cdots & \alpha_{n-2}\\
\vdots & \ddots & \ddots &  \ddots & \vdots\\
\alpha_2 & \alpha_3 &\ldots & \ldots & \alpha_1
\end{pmatrix}
\end{equation}

We also consider the following generalizations of this notion.

Matrix $A\in\Rmn$ will be called a {\em rectangular circulant}, if
$a_{ij}=a_{ps}$ whenever $p=i+t(\modd\; m)$ and $s=j+t(\modd\; n)$,
for all $i,j,t$. When $m=n$, this is an ordinary circulant
given by \eqref{def-circulant}.

Matrix $A\in\Rmn$ will be called a {\em block $k\times k$ circulant} when there exist
scalars $\alpha_1,\ldots,\alpha_k$ and a block decomposition
$A=(A_{ij}),\ i,j=1,\ldots,k$ such that $A_{ij}=\alpha_d E_{ij}$ if
$j-i=d(\modd\; k)$, where all entries of blocks $E_{ij}$ are equal to $\bunity$.

$A\in\Rmn$ is called {\em $d$-periodic} when
$a_{ij}=a_{is}$ if
$(s-j)\modd n$ is a multiple of $d$, and when $a_{ji}=a_{si}$
if $(s-j)\modd m$ is a multiple of $d$.

We give an example of $6\times 9$ rectangular circulant $A$:
\begin{equation*}
A=
\begin{pmatrix}
0 & 1 & 2 & 0 & 1 & 2 & 0 & 1 & 2\\
2 & 0 & 1 & 2 & 0 & 1 & 2 & 0 & 1\\
1 & 2 & 0 & 1 & 2 & 0 & 1 & 2 & 0\\
0 & 1 & 2 & 0 & 1 & 2 & 0 & 1 & 2\\
2 & 0 & 1 & 2 & 0 & 1 & 2 & 0 & 1\\
1 & 2 & 0 & 1 & 2 & 0 & 1 & 2 & 0
\end{pmatrix}.
\end{equation*}
This example provides evidence that a rectangular
$m\times n$ circulant consists of ordinary $d\times d$
circulant blocks where $d=$g.c.d.$(m,n)$. In particular,
it is $d$-periodic.
Also, there exist permutation matrices
$P$ and $Q$ such that $PAQ$ is a
block circulant:
\begin{equation*}
B=
\begin{pmatrix}
0 & 0 & 0 & 1 & 1 & 1 & 2 & 2 & 2\\
0 & 0 & 0 & 1 & 1 & 1 & 2 & 2 & 2\\
2 & 2 & 2 & 0 & 0 & 0 & 1 & 1 & 1\\
2 & 2 & 2 & 0 & 0 & 0 & 1 & 1 & 1\\
1 & 1 & 1 & 2 & 2 & 2 & 0 & 0 & 0\\
1 & 1 & 1 & 2 & 2 & 2 & 0 & 0 & 0
\end{pmatrix}
\end{equation*}
We formalize these observations in the following.

\begin{proposition}
\label{p:circul}
Let $A\in\Rmn$ be a rectangular circulant,
and let $d=$g.c.d.$(m,n)$.
\begin{itemize}
\item[1.] $A$ is $d$-periodic.
\item[2.] There exist permutation matrices $P$ and $Q$ such that $PAQ$
is a block $d\times d$ circulant.
\end{itemize}
\end{proposition}
\begin{proof}
1. There are integers $t_1$ and $t_2$
such that $d=t_1m+t_2n$. Using the definition of
rectangular circulant we obtain $a_{ij}=a_{is}$, if
$s=j+t_1m (\modd\; n)$, and hence if $s=j+d(\modd\; n)$.
Analogously for rows, we obtain that
$a_{ji}=a_{si}$, if $s=j+t_2n (\modd\; m)$, and hence if
$s=j+d(\modd\; m)$.\\
\noindent
2. As $A$ is $d$-periodic, all rows such that $i+d=j(\modd\; m)$
are equal, so that $\{1,\ldots,m\}$ can be divided in $d$ groups
each with $m/d$ indices, in such a way that $A_{i\cdot}=A_{j\cdot}$
if $i$ and $j$ belong to the same group. We can find a permutation
matrix $P$ such that $A'=PA$ will have rows
$A'_{1\cdot}=\ldots=A'_{d\cdot}=A_{1\cdot}$,
$A'_{d+1\cdot}=\ldots=A'_{2d\cdot}=A_{2\cdot}$, and so on.
Analogously we can find a permutation matrix $Q$ such that
$A''=PAQ$ will have columns
$A''_{\cdot 1}=\ldots=A''_{\cdot d}=A'_{\cdot 1}$,
$A''_{\cdot d+1}=\ldots=A''_{\cdot 2d}=A'_{\cdot 2}$, and so on.
Then $A''$ has blocks $(A''_{ij})$ for $i,j=1,\ldots,d$.
of dimension $n/d\times m/d$, where
$A''_{ij}=a_{ij}E_{ij}$, and $E_{ij}$ has all entries $1$.
As $A$ is $d$-periodic, it can be shown that its submatrix
extracted from the first $d$ rows and columns, is a circulant.
Hence $A''$ is a block circulant.
\end{proof}

When g.c.d.$(m,n)=1$, the rectangular circulant is $1$-periodic,
i.e., a constant matrix.

\begin{proposition}
\label{p:critcirc}
Let $A\in\Rnn$ be a definite visualized matrix which admits
block decomposition \eqref{munublocks},
and $r\geq T(A)$. Let $\critgraph_{\mu},\critgraph_{\nu}$ be two (possibly
equal) components of $C(A)$, and $d=g.c.d.(\gamma_{\mu},\gamma_{\nu})$.
\begin{itemize}
\item[1.] $\Tilde{A}_{\mu\nu}^{(r)}$ is a rectangular circulant
(which is a circulant if $\mu=\nu$).
\item[2.] For any critical $\mu$ and $\nu$, there is
a permutation $P$ such that
$(P^T\Tilde{A}P)^{(r)}_{\mu\nu}$ is a block
$d\times d$ circulant matrix.
\item[3.] If $r$ is a multiple of $\caa$,
then $\Tilde{A}^{(r)}_{\mu\mu}$
are circulant Kleene stars, where all off-diagonal entries are
strictly less than $1$.
\end{itemize}
\end{proposition}
\begin{proof}
1.: Using Eqn. \eqref{e:bellman3} we see that for all $(i,j)$ and $(k,l)$ such that
$k=i+t(\modd \gamma_{\mu})$ and
$l=j+t(\modd \gamma_{\nu})$,
$$\Tilde{a}_{s_kt_l}^{(r)}=\Tilde{a}_{s_it_l}^{(r+t)}=
\Tilde{a}_{s_it_j}^{(r)}.$$

\msn
2.: If $\mu=\nu$ then $P=I$, and if $\mu\neq\nu$ then $P$
is any permutation matrix such that its ``subpermutations''
for $N_{\mu}$ and $N_{\nu}$ are given by $P$ and $Q$ of
Proposition \ref{p:circul}.

\msn
3.: Part 1 shows that $\Tilde{A}^{(r)}_{\mu\mu}$ are
circulants for any $r\geq T(A)$ and critical $\mu$.
If $r$ is a multiple of $\caa$, then $\Tilde{A}_{\mu\mu}^{(r)}$
are submatrices of $\Tilde{A}^{\gamma}=Q(\Tilde{A}^{\gamma})$
and hence of $(\Tilde{A}^{\gamma})^*$. This implies,
using \eqref{kleene2}, that they are Kleene stars.
As the $\mu$th component of $\crit(\Tilde{A})$ is just
a cycle of length $\caa_{\mu}$, the corresponding
component of $\crit(\tilde{A}^{\caa})$ consists of
$\caa_{\mu}$ loops, showing that the off-diagonal
entries of $\Tilde{A}^{(r)}_{\mu\mu}$ are strictly less than $1$.
\end{proof}


\section{Attraction cones}

\subsection{A system for attraction cone}

Let $A\in\Rnn$ and $\lambda(A)=\bunity$.
Recall that {\em attraction cone} $\operatorname{Attr}(A,t)$, where
$t\geq 1$ is an integer,
is the set which consists of
all vectors $x$, for which there exists an integer $r$ such that
$A^r\otimes x=A^{r+t}\otimes x$, and hence
this is also true for all integers greater than or equal to $r$.
We have shown above, see Subsect. 4.3, that $\operatorname{Attr}(A,t)$ is solution
set to the critical subsystem of the system 
$A^{r+t}\otimes x=A^r\otimes x$, for any $r\geq T_c(A)$.
Next we show how
the specific circulant structure of $A^r$ at $r\geq T_c(A)$
can be exploited, to derive a more concise system of equations
for the attraction cone $\Attr(A,1)$.
Due to Theorem \ref{t:maxentries}
the core matrix
$A^{\Core}=\{\alpha_{\mu\nu}\mid\mu,\nu=1,\ldots,\critcomps\}$,
and its Kleene star
$(A^{\Core})^*=\{\alpha_{\mu\nu}^*\mid\mu,\nu=1,\ldots,\critcomps\}$
will be of special importance.
We introduce
the notation
\begin{equation}
\label{stnot}
\begin{split}
M_{\nu}^{(r)} (i)&=\{j\in N_{\nu}\mid a_{ij}^{(r)}=
\alpha_{\mu\nu}^*\},\ i\in N_{\mu},\ \forall\nu: \crit_{\nu}\neq \crit_{\mu},\\
K^{(r)}(i)&=\{t>\critnodes\mid a_{it}^{(r)}=\alpha_{\mu\nu(t)}^*\},\
i\in N_{\mu},
\end{split}
\end{equation}
where $\crit_{\mu}$ and $\crit_{\nu}$ are s.c.c. of
$\crit(A)$, $N_{\mu}$ and $N_{\nu}$
are their node sets, and $\nu(t)$ in the second
definition denotes the index of the non-critical component which consists of
the node $t$. The sets $M_{\nu}^{(r)}(i)$ are non-empty
for any $r\geq T_c(A)$, due to Theorem \ref{t:maxentries}.

The results of Subsect. \ref{ss:critcirc} lead to the following
properties of $M_{\nu}^{(r)}(i)$ and $K^{(r)}(i)$.

\begin{proposition}
\label{stprops}
Let $r\geq T_c(A)$ and $\mu,\nu\in\{1,\ldots,\critcomps\}$.
\begin{itemize}
\item[1.] If $[i]\to_t[j]$ and $i,j\in N_{\mu}$ then
$M_{\nu}^{(r+t)}(i)=M_{\nu}^{(r)} (j)$
and $K^{(r+t)}(i)=K^{(r)}(j)$.
\item[2.] Each $M_{\nu}^{(r)}(i)$ is the union of
some cyclic classes of $\crit_{\nu}$.
\item[3.] Let $i\in N_{\mu}$ and
$d=g.c.d.(\gamma_{\mu},\gamma_{\nu})$.
Then, if $[p]\subseteq M_{\nu}^{(r)}(i)$
and $[p]\to_d [s]$ then $[s]\subseteq M_{\nu}^{(r)}(i)$.
\item[4.] Let $i,j\in N_{\mu}$ and $p,s\in N_{\nu}$.
Let $[i]\to_t[j]$ and $[p]\to_t[s]$. Then
$[p]\subseteq M_{\nu}^{(r)}(i)$ if and only if
$[s]\subseteq M_{\nu}^{(r)}(j)$.
\end{itemize}
\end{proposition}

Next we establish the
cancellation rules which will enable us to write out
a concise system of equations for the attraction cone $\Attr(A,1)$.

If $a<c$, then
\begin{equation}
\label{simplcanc}
\{x\colon ax\oplus b=cx\oplus d\}\ =\ \{x\colon b=cx\oplus d\}.
\end{equation}

Now consider a system
of equations over max algebra:
\begin{equation}
\label{e:sysgen1}
\bigoplus_{i=1}^n a_{1i} x_i\oplus c_1=\bigoplus_{i=1}^n a_{2i} x_i \oplus c_2=\ldots=
\bigoplus_{i=1}^n a_{ni} x_i\oplus c_n.
\end{equation}
Suppose that $\alpha_1,\ldots,\alpha_n\in\R$ are such that
$a_{li}\leq \alpha_i$ for all $l$ and $i$, and
$S_l=\{i\mid a_{li}=\alpha_i\}$ for $l=1,\ldots,n$.
Let $S_l$ be such that
$\bigcup_{l=1}^n S_l=\{1,\ldots,n\}$. Repeatedly applying
the elementary cancellation law described above, we obtain that
\eqref{e:sysgen1} is equivalent to
\begin{equation}
\label{e:syscanc1}
\bigoplus_{i\in S_1}\alpha_i x_i\oplus c_1=
\bigoplus_{i\in S_2}\alpha_i x_i\oplus c_2=\ldots=
\bigoplus_{i\in S_n}\alpha_i x_i\oplus c_n.
\end{equation}
We will refer to the equivalence between \eqref{e:sysgen1}
and \eqref{e:syscanc1},
which we acknowledge to \cite{Dok:08},
as to {\em chain cancellation}.

Using notation \eqref{stnot} and
Proposition \ref{stprops} we formulate the
main result of the paper.

\begin{theorem}
\label{mainres}
Let $A\in\Rnn$ be a visualized matrix
and $r\geq T(A)$ be a multiple of $\caa$. Then the system
$A^r\otimes x=A^{r+1}\otimes x$ is equivalent to
\begin{equation}
\label{mainsys}
\begin{split}
&\bigoplus_{k\in[i]}
x_k\oplus
\bigoplus_{\crit_{\nu}\neq \crit_{\mu}}\alpha_{\mu\nu}^*
\left(\bigoplus_{k\in M_{\nu}^{(r)}(i)} x_k\right)
\oplus\bigoplus_{t\in K^{(r)}(i)} \alpha_{\mu\nu(t)}^* x_t=\\
&=\bigoplus_{k\in[j]}x_k\oplus\bigoplus_{\crit_{\nu}\neq
\crit_{\mu}}\alpha_{\mu\nu}^*
\left(\bigoplus_{k\in M_{\nu}^{(r)}(j)} x_k\right)
\oplus\bigoplus_{t\in K^{(r)}(j)} \alpha_{\mu\nu(t)}^* x_t,
\end{split}
\end{equation}
where $\crit_{\mu}$ is the component of
$\crit(A)$ which contains both $[i]$ and $[j]$, and $[i]$ and $[j]$ range
over all pairs of cyclic classes such that $[i]\to_1[j]$.
\end{theorem}
\begin{proof} By Proposition \ref{onlycrit}
$A^r\otimes x=A^{r+1}\otimes x$ is equivalent to its critical subsystem.
Consider a critical component of $A^r\otimes x=A^{r+1}\otimes x$:
\begin{equation}
\label{e:critcomp}
\bigoplus_k a_{ik}^{(r)} x_k=
\bigoplus_k a_{ik}^{(r+1)} x_k,\ i=1,\ldots,\critnodes.
\end{equation}
Consider $j$ such that $[i]\to_1 [j]$.
Then by Proposition \ref{p:bellman2},
$a_{ik}^{(r+1)}=a_{jk}^{(r)}$,
hence the critical subsystem of
$A^r\otimes x=A^{r+1}\otimes x$ is as follows:
\begin{equation}
\label{pre-form}
\bigoplus_k a_{ik}^{(r)} x_k=\bigoplus_k a_{jk}^{(r)} x_k,\ \forall i,j:\ [i]\to_1[j].
\end{equation}
The number of non-identical equalities in \eqref{e:critcomp}
and \eqref{pre-form} is equal to the total number of cyclic classes.

Proposition \ref{p:critcirc}, part 3, implies that
all principal submatrices of $A^r$ extracted from critical
components have a circulant block structure.
In this structure,
all entries of the diagonal blocks are equal to $\bunity$,
and the entries of all off-diagonal blocks are strictly
less than $\bunity$. Hence we can apply the chain
cancellation (equivalence
between \eqref{e:sysgen1} and \eqref{e:syscanc1})
and obtain the first terms on both sides of
\eqref{mainsys}. By Theorem \ref{t:maxentries} each block $A_{\mu\nu}$ contains an entry equal to
$\alpha_{\mu\nu}^*$. For a non-critical $\nu(t)$, this readily implies that the corresponding ``subcolumn''
$A_{\mu\nu(t)}$ contains an entry $\alpha_{\mu\nu(t)}^*$. Applying
the chain cancellation we obtain the last terms
on both sides of \eqref{mainsys}. Due to the
block circulant structure of $A_{\mu\nu}$ with both
$\mu$ and $\nu$ critical, see Proposition \ref{p:critcirc} or
Proposition \ref{stprops},
we see that each column of such block also contains
an entry equal to $\alpha_{\mu\nu}^*$. Applying the
chain cancellation we obtain the remaining terms in \eqref{mainsys}.
\end{proof}

As $\Attr(A,t)=\Attr(A^t,1)$,
system \eqref{mainsys} also describes more general attraction cones,
it only amounts to substitute $\crit(A^t)$ for $\crit(A)$ and the entries of
$((A^t)^{\Core})^*$ for $\alpha_{\mu\nu}^*$
(the dimension of this matrix will be different
in general, see Proposition \ref{butkovic} part 3).

We note that the system for $\Attr(A,1)$ naturally
breaks into several chains of equations corresponding to
the s.c.c. of $\crit(A)$. If we start with \eqref{pre-form},
it can be equivalently written as $R\otimes x=H\otimes y$ where
$R$ is the factor in $CSR$-representation, and
$H\in\R^{\critnodes\times\critcomps}$
is a Boolean matrix with entries
\begin{equation}
\label{h-def}
h_{i\mu}=
\begin{cases}
1, &\text{if $i\in N_{\mu}$},\\
0, &\text{otherwise}.
\end{cases}
\end{equation}
We can apply cancellation as described in the proof
of Theorem \ref{mainres}, to get rid of redundant terms
on the left-hand side of the two-sided system.

If $\crit(A)$ is strongly connected then $H$ is a vector
of all ones, and the two-sided system $R\otimes x=H\otimes y$ becomes
essentially one-sided. We treat this case in the next
subsections.

\subsection{Extremals of attraction cones}
System \eqref{mainsys} in general consists of several
chains of equations corresponding to s.c.c. of $\crit(A)$.
Each chain is of the form
\begin{equation}
\label{e:chain}
\bigoplus_{i\in T_1} a_ix_i=\bigoplus_{i\in T_2} a_ix_i=\ldots=
\bigoplus_{i\in T_m} a_ix_i,
\end{equation}
where $T_1\cup\ldots\cup T_m=\{1,\ldots,n\}$ and $a_i$
come from the entries of $(A^{\Core})^*$.

Here we consider only the case of strongly connected $\crit(A)$,
i.e., only one chain. By scaling $y_i=a_ix_i$ we obtain
\begin{equation}
\label{e:chain1}
\bigoplus_{i\in T_1} y_i=\bigoplus_{i\in T_2} y_i=\ldots=
\bigoplus_{i\in T_m} y_i,
\end{equation}

Note that when the critical graph
is not strongly connected,
we have several chains of equations
and the coefficients of \eqref{e:chain} in general
cannot be scaled to get \eqref{e:chain1}
for each chain at the same time.

By $e^i$ we denote the vector which has the
$i$th coordinate and all the rest
equal to $0$. Vector $y\in\R^n$ will
be called {\em scaled} if $\bigoplus_{i=1}^n y_i=1$,
and set $S$ is called scaled if it consists of scaled vectors.
We say that $V\subseteq\R^n$ is {\em generated} by $S\subseteq\R^n$
(also, $S$ is a generating set of $V$)
if $S\subseteq V$ and
for any $x\in V$ there exist $y^1,\ldots,y^l\in S$
and $\alpha_1,\ldots,\alpha_l\in\R$
such that $x=\bigoplus_{i=1}^l \alpha_i y^i$.

We investigate extremal solutions of
\eqref{e:chain1}: a solution $x$
is called {\em extremal} if $x=y\oplus z$ for two other solutions
$y,z$ implies that $x=y$ or $x=z$ \cite{BSS-07}. The
following can be deduced from the results of \cite{BH-84,BSS-07}.

\begin{proposition}
The solution set of any finite system of max-linear equations
has a finite generating set. In particular, it is
generated by extremal solutions, and any set of
scaled generators for the solution set contains all
scaled extremal solutions.
\end{proposition}

In the next proposition we show that
extremal solutions of \eqref{e:chain1} can
have only $0$ and $1$ components.

\begin{proposition}
\label{extr01}
Let $y$ be a scaled solution of \eqref{e:chain1} and
let $0<y_i<1$ for some $i$.
Then $y$ is not an extremal.
\end{proposition}
\begin{proof}
Let $K^<:=\{i\mid 0<y_i<1\}$ and $K^1:=\{i\mid y_i=1\}$, and
define vectors
$v^0$ and $v^1(k)$ for each $k\in K^<$ by
\begin{equation}
v^0_i=
\begin{cases}
1, & \text{if $i\in K^1$}\\
0, & \text{otherwise}
\end{cases},
\quad
v^1_i(k)=
\begin{cases}
1, & \text{if $i\in K^1\cup\{k\}$}\\
0, & \text{otherwise}
\end{cases}.
\end{equation}
Observe that both $v^0$ and $v^1(k)$ for any
$k\in K^<$, are solutions to
\eqref{e:chain1}, different from $y$. We have:
\begin{equation}
y=v^0\oplus\bigoplus_{k\in K^<} y_k\cdot v^1(k),
\end{equation}
hence $y$ is not an extremal.
\end{proof}

Let $T=(t_{ij})$ be the $m\times n$ $0-1$ matrix defined by
\begin{equation}
\label{def-tij}
t_{ij}=
\begin{cases}
1, & \text{if $j\in T_i$,}\\
0, & \text{otherwise},
\end{cases}
\end{equation}
where $T_i$ are from \eqref{e:chain1}.

A subset $K\subseteq\{1,\ldots,n\}$ is called a {\em covering} of $T$ if
each $T_i$ contains an index from $K$. The following is immediate.
\begin{proposition}
\label{p:simple}
A scaled vector $y$ is a solution of \eqref{e:chain1} if and only if
$K^1:=\{i\mid y_i=1\}$ is a covering of $T$.
\end{proposition}

A covering $K$ is called {\em minimal} if it does not contain any
proper subset which is also a covering.

A covering $K$ will be called {\em nearly minimal}
if it contains no more than one
proper subcovering $K'$. Observe that then
the complement $K\backslash K'$ consists of just
one index. Hence, a covering is nearly minimal if and only if
there may exist
no more than one $i\in K$ such
that $K\backslash\{i\}$ is a covering.

\begin{proposition}
\label{mainres-chains}
Extremal solutions of \eqref{e:chain1} are precisely $v^K=\bigoplus_{i\in K} e_i$,
where $K$ is a nearly minimal covering of $T$.
\end{proposition}
\begin{proof}
If a covering $K$ is not nearly minimal,
then there exist $i$ and $j$ such that
$K[i]:=K\backslash\{i\}$ and $K[j]:=K\backslash\{j\}$ are both coverings of $T$.
Then $v^{K[i]}$ and $v^{K[j]}$ are both solutions and $v=v^{K[i]}\oplus v^{K[j]}$
hence $v^{K}$ is not extremal.

Conversely, if $v^{K}$ is not extremal,
then there exist $y\neq v^{K}$ and $z\neq v^{K}$
such that $v^K=y\oplus z$. Evidently $y\leq v^K$ and
$z\leq v^K$. By Proposition \ref{extr01}
we can represent
$y$ and $z$ as combinations of $0-1$ solutions of \eqref{e:chain1}.
These solutions
correspond to coverings, which must be proper subsets of $K$.
At least two of these
coverings must be different from each other,
hence $K$ is not nearly minimal.
\end{proof}

Thus, the problem of finding all nearly minimal coverings
of a $0-1$ matrix
is equivalent to the problem of finding all
extremal solutions of \eqref{e:chain1}.

The following
case applies if the critical graph is strongly connected
and occupies all nodes.

\begin{corollary}
\label{c:disjoint}
Let $T_1,\ldots,T_m$ be pairwise disjoint,
then the scaled extremals are precisely
all vectors $v^S=\bigoplus_{i\in S} e^i$, where $S$ is an
index set which contains
exactly one index from each set $T_i$.
\end{corollary}
\begin{proof}
Any such set $S$ forms a minimal covering of $T$, and it can
be shown that the
solution set of \eqref{e:chain1} is generated by $v^S$, so
there are no more scaled extremals (or nearly minimal coverings).
\end{proof}

\subsection{An algorithm for finding the coefficients
of an attraction system}\label{ss:str}

Coefficients of the system of equations which defines attraction
cone are determined by the entries of $(A^{\Core})^*$ which can be
found in $O((\critcomps+\noncrit)^3)$ operations, where $\critcomps$
is the number of s.c.c. of $\crit(A)$ and $\noncrit$ is the number
of non-critical nodes. However it remains to find the places where
these coefficients appear, i.e., the sets $M_{\nu}^{(r)}(i)$ and
$K^{(r)}(i)$ for $i=1,\ldots,\critnodes$. Solving this problem, we
get another polynomial method for computing the coefficients of
\eqref{mainsys}.

Here we restrict our attention to the case
when $\crit(A)$ is strongly connected.
In this case there are no second terms on both sides
of \eqref{mainsys} and we need only
$K^{(r)}(i)$. The digraph
$\Digr(A^{\Core})$ associated with
the matrix $A^{\Core}$ consists of one critical node which
corresponds to the whole
$\crit(A)$ and will be
denoted by $\mu$, and $\noncrit$ non-critical nodes $\nu(t)$,
for $t>c$. The core matrix is of the form
\begin{equation}
\label{core-def}
A^{\Core}=
\begin{pmatrix}
1 & h\\
g & B
\end{pmatrix},
\end{equation}
the entries $\alpha_{\mu\nu}$ and the entries of
$h$, $g$ and $B=(b_{\nu(s),\nu(t)})$ are given by
\begin{equation}
\label{acentries}
\begin{split}
\alpha_{\mu\mu}&=\bunity,\\
h_{\nu(t)}=\alpha_{\mu\nu(t)}&=\max_{k=1}^c a_{kt},\
g_{\nu(t)}=\alpha_{\nu(t)\mu}=\max_{k=1}^c a_{tk},\ t>c,\\
b_{\nu(s)\nu(t)}=\alpha_{\nu(s)\nu(t)}&=a_{st},\ s>c,\ t>c.
\end{split}
\end{equation}
Denote by $[\to_m i]$ the cyclic class $[j]$ such that
$[j]\to_m[i]$.
For each $t=1,\ldots,\critnodes$,
we initialize Boolean $\critnodes$-vectors $P_t$ by
\begin{equation}
\label{pdef}
P_t(i)=
\begin{cases}
1, &\text{if $[\to_1 i]\cap\arg\max\limits_{k=1}^{\critnodes}
a_{kt}\ne\emptyset$}\\
0, &\text{otherwise.}
\end{cases}
\end{equation}
$P_t$ encode the Boolean information associated with
$h$ (shifted cyclically by $1$).

Further we compute the Kleene star
of the non-critical submatrix $B:=A_{MM}$, where
$M$ denotes the set of non-critical nodes, and store
the information on the lengths
of paths with maximal weight and length not exceeding
$\noncrit$ in sets $U_{st}$ associated to
each entry of $B$. To compute these sets
we use the formula
\begin{equation}
\label{noncrit-star}
B^*=I\oplus B\oplus\ldots\oplus B^{\noncrit-1},
\end{equation}
where for each entry of this matrix series we
find the arguments of maxima.

To combine the information associated
with $h$ and $B^*$, we recall the max-algebraic version of
bordering method \cite{Car-71}, which computes
\begin{equation}
\label{bordering}
(A^C)^*=
\begin{pmatrix}
\bunity & h^T\\
g & B
\end{pmatrix}^*
=
\begin{pmatrix}
\bunity & h^T\otimes B^*\\
B^*\otimes g & B^*\oplus B^*\otimes g\otimes h^T\otimes B^*
\end{pmatrix},
\end{equation}
where $h,g\in\R^{\noncrit}$. Note that all information that we need
for system \eqref{mainsys}, is in the entries of $h^T\otimes B^*$
and in the indices of equations of the system where the entries of
$h^T\otimes B^*$ appear. Computing $(h^T\otimes B^*)_i$ means in
particular obtaining the ``winning'' indices
\begin{equation}
\label{wtdef}
W_t=\arg\max_{s>c} h_{\nu(s)} b_{\nu(s)\nu(t)}^*.
\end{equation}
After that, the idea is to combine $P_s$ with $U_{st}$ for all $s\in
W_t$ and unite the obtained indices. More precisely, for each number
$m$ stored in $U_{st}$ we define the shifted Boolean vector
$P_s^{\to m}$ by
\begin{equation}
\label{pmdef1}
P_s^{\to m}(i)=
\begin{cases}
1, &\text{if $[\to_{m+1} i]\cap\arg\max\limits_{k=1}^{\critnodes}
a_{ks}\ne\emptyset$}\\
0, &\text{otherwise.}
\end{cases}
\end{equation}
Equivalently,
\begin{equation}
\label{pmdef2}
P_s^{\to m}(i)=1\Leftrightarrow P_s(j)=1\ \text{and}\ [j]\to_m[i].
\end{equation}
After that, we define
\begin{equation}
\label{pmunu1def}
G_t:=\bigvee_{s\in W_t}\bigvee_{m\in U_{st}} P_s^{\to m}.
\end{equation}

\begin{proposition}
\label{p:gti}
Let $t>c$, and let $r\geq T(A)$ be a multiple of
$\caa$. Then for all $i\leq\critnodes$,
$t\in K^{(r)}(i)$ if and only if $G_t(i)=1$.
\end{proposition}
\begin{proof}
$G_t(i)=1$ if and only if there exist $s\in W_t$ and $m\in U_{st}$
such that $P_s^{\to m}(i)=1$. Then there exists a path
$\pi_1\circ\tau\circ\pi_2$ where $\pi_1$ starts in $i$, belongs to
$\crit(A)$ and has length $-m-1(\modd \;\gamma)$, $\tau$ is an edge
which attains $h_{\nu(s)}=\max_{l=1}^{\critnodes} a_{ls}$, and
$\pi_2$ is entirely non-critical, has length $m$, weight
$b_{\nu(s)\nu(t)}^*$, and connects $s$ to $t$. The weight of
$\pi_1\circ\tau\circ\pi_2$ is equal to $1\cdot h_s\cdot
b_{st}^*=\alpha_{\mu\nu(t)}^*$, and the length is a multiple of
$\caa$, meaning that $a_{it}^{(r)}=\alpha_{\mu\nu(t)}^*$ and $t\in
K^{(r)}(i)$.

The other way around, let $t\in K^{(r)}(i)$. Then there exists a
path $P$ of length $r$ that connects $i$ to $t$ and has weight
$\alpha_{\mu\nu(t)}^*=\bigoplus_s h_{\nu(s)} b_{\nu(s)\nu(t)}^*$. We
can decompose $P=\pi_1'\circ\tau'\circ\pi'_2$, where $\tau'$ is an
edge connecting a critical node to a non-critical node, and $\pi'_2$
has only non-critical nodes. Obviously
$w(\tau'\circ\pi'_2)\leq\bigoplus_s h_{\nu(s)} b_{\nu(s)\nu(t)}^*=
\alpha_{\mu\nu(t)}^*$. But $w(P)=\alpha^*_{\mu\nu(t)}$, hence
$w(\pi'_1)=1$ so that $\pi'_1$ entirely belongs to $\crit(A)$, and
$w(\tau'\circ\pi'_2)=\alpha^*_{\mu\nu(t)}$. In particular,
$w(\tau')=h_{\nu(s)}$ and $w(\pi'_2)=b^*_{\nu(s)\nu(t)}$ for certain
$s>c$. From this we conclude that $G_t(i)=1$.
\end{proof}

Summarizing above said, we have the following algorithm for
computing the coefficients of \eqref{mainsys} in the case when
$\crit(A)$ is strongly connected. Recall that in this case there is
no second term on both sides of \eqref{mainsys}. The computation of
coefficients of the third term includes the computation of $h\otimes
B^*$ and the sets $K^{(r)}(i)$ for $i\leq\critnodes$ (in fact we can
improve the algorithm since only $\caa$ of them are different).

\vskip 1cm
\noindent{\bf ALGORITHM 1.}\ {\em  Compute the coefficients of \eqref{mainsys} if
$\crit(A)$ is strongly connected.}
\msn

{\bf Input.} Visualized matrix $A$, critical graph $\crit(A)$
which is strongly connected and the cyclic classes of $\crit(A)$.\\
{\bf 1.} Compute $h$ and initialize $P_t$ for $t>\critnodes$. This
takes
$\critnodes\noncrit$ operations.\\
{\bf 2.} Compute $B^*$ and sets $U_{st}$ for all $s,t>c$. It takes
$O(\noncrit^4)$ operations
(see \eqref{noncrit-star}).\\
{\bf 3.} Compute $hB^*$ and $G_t$ for $t>c$, by \eqref{wtdef},
\eqref{pmdef2} and \eqref{pmunu1def}. Computation of $hB^*$ and
$W_t$ by \eqref{wtdef} requires $O(\noncrit^2)$ operations,
computation of shifted Boolean vectors is $O(\critnodes
\noncrit^2)$, and the
conjunction \eqref{pmunu1def} takes $O(\critnodes \noncrit^3)$ operations.\\
{\bf 4.} Compute $K^{(r)}(i)$ using Proposition
\ref{p:gti}. This requires $\critnodes\noncrit$
operations.

\msn
As the overall complexity does not exceed $O(\noncrit^4)+O(\critnodes\noncrit^3)=
O(n\noncrit^3)$ operations,
we conclude the following.

\begin{proposition}
\label{t:ccsmart}
Let $A\in\Rnn$ be visualized, $\crit(A)$ be strongly connected,
$\noncrit$ be the number of non-critical nodes, and
suppose we know $\crit(A)$ and all cyclic classes.
Then Algorithm 1 computes the coefficients of the attraction
system in no more than $O(\noncrit^3 n)$ operations.
\end{proposition}

It is also important that the eigenvalue and an eigenvector of
irreducible matrix can be computed by the policy iteration algorithm
of \cite{Coc-98}, which is very fast in practice. After that,
$\crit(A)$ and the cyclic classes can be computed in $O(n^2)$ time.
Thus we are led to an efficient method of computing the coefficients
in the case when $A$ is irreducible and $\crit(A)$ is strongly
connected, especially in the case when the number of non-critical
nodes is small. Note that the case of irreducible $A$ and strongly
connected $\crit(A)$ is generic when matrices $A$ are real and
generated at random. Also, in this generic case it almost never
happens that maxima in blocks or among the weights of paths are
achieved twice, which means that we do not need to assign Boolean
vectors or sets to each entry. In this case the total number of
operations in the algorithm is reduced to
$O(\noncrit^3)+O(\noncrit\critnodes)$.

When $\crit(A)$ is not strongly connected, the bordering method
\eqref{bordering} can be used to obtain an
algorithm which operates only
with the entries of $A^{\Core}$. However, the complexity of
operations with indices and Boolean numbers significantly
increases in that general case.

\section{Examples}

\subsection{Square multiplication}

In this subsection we will examine the problems that can be solved
by matrix squaring on $9\times 9$
real matrix over the {\em max-plus semiring}:
\begin{equation*}
A=
\begin{pmatrix}
-1 & 0 & -1 & -1 & -9 & -7 & -10  & -4 & -8\\
0 & -1 & 0 & -1 & -10 & -1 & -10  & -9 & -4\\
-1 & -1 & -1 & 0 & -2 & -3 & -2   & -6 & -6\\
0 & -1 & -1 & -1 & -10 & -6 & -10 & -6 & -1 \\
-10 & -2 & -8 & -1 & -1 & 0 & -1  & -10 & -1\\
-5 & -5 & -10 & -9 & -1 & -1 & 0  & -3 & -6\\
-9 & -10 & -7 & -10 & 0 & -1 & -1 & -8 & -8\\
-75 & -80 & -77 & -83 & -80 & -77 & -82 & -2 & -0.5\\
-84 & -81 & -77 & -80 & -78 & -77 & -78 & -0.5 & -2
\end{pmatrix}
\end{equation*}

The corresponding max-times example is obtained by, e.g., taking
{\em exponents} of the entries.

The critical graph of $A$, see Figure 3, has two s.c.c.: $C_1$ with nodes 
$N_1=\{1,2,3,4\}$ and $C_2$ with nodes $N_2=\{5,6,7\}$.
The cyclicity of $C_1$ is $\gamma_1=2$ and the cyclicity
of $C_2$ is $\gamma_2=3$, so the cyclicity of $\critgraph(A)$ is
$\gamma=\lcm(2,3)=2\times 3=6$.

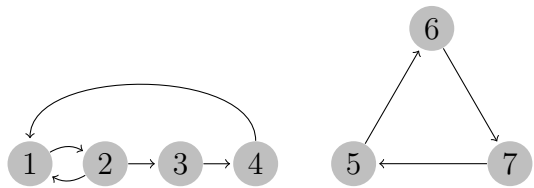
\begin{figure}[h]
\begin{tabular}{cccc}
\begin{tikzpicture}[shorten >=1pt,->]
  \tikzstyle{vertex1}=[circle,fill=black!25,minimum size=17pt,inner sep=1pt]
\foreach \name/\x in {1/1, 2/2, 3/3, 4/4}
    \node[vertex1] (\name) at (\x,0) {$\name$};
\draw (2) -- (3);
\draw (3) -- (4);
\draw (1) .. controls +(30:0.5cm) and +(150:0.5cm) .. (2);
\draw (2) .. controls +(-150:0.5cm) and +(-30:0.5cm) .. (1);
\draw (4) .. controls +(90:1.3cm) and +(90:1.3cm) .. (1);
\end{tikzpicture}&&
\begin{tikzpicture}[shorten >=1pt,->]
  \tikzstyle{vertex1}=[circle,fill=black!25,minimum size=17pt,inner sep=1pt]

\foreach \name/\angle/\text in {5/-150/5, 6/90/6, 
                                  7/-30/7}
    \node[vertex1,xshift=0cm,yshift=0cm] (\name) at (\angle:1.2cm) {$\text$};
\foreach \from/\to in {5/6,6/7,7/5}
    \draw (\from) -- (\to);
\end{tikzpicture}
\end{tabular}
\caption{The critical graph of $A$}
\end{figure}

The matrix can be decomposed into blocks
\begin{equation*}
A=
\begin{pmatrix}
A_{11} & A_{12} & A_{1M}\\
A_{21} & A_{22} & A_{2M}\\
A_{M1} & A_{M1} & A_{MM}
\end{pmatrix},
\end{equation*}
where the submatrices $A_{11}$ and $A_{22}$
correspond to two s.c.c. $C_1$ and $C_2$ of
$\critgraph(A)$, see Figure 3. They equal 

\begin{equation*}
A_{11}=
\begin{pmatrix}
-1 & 0 & -1 & -1 \\
0 & -1 & 0 & -1 \\
-1 & -1 & -1 & 0 \\
0 & -1 & -1 & -1
\end{pmatrix},\quad
A_{22}=
\begin{pmatrix}
-1 & 0 & -1\\
-1 & -1 & 0\\
0 & -1 & -1
\end{pmatrix},
\end{equation*}
and $A_{MM}$ is the non-critical principal submatrix
\begin{equation*}
A_{MM}=
\begin{pmatrix}
-2 & -0.5\\
-0.5 & -2
\end{pmatrix}.
\end{equation*}
The submatrices $A_{12}$, $A_{21}$, $A_{1M}$ and $A_{2M}$
are composed of randomly taken numbers from $-1$ to $-10$, and
$A_{M1}$ and $A_{M2}$ are composed of randomly taken numbers from
$-75$ to $-85$.

It can be checked that the powers of $A$ become periodic after $T(A)=154$.

We will consider the following instances of problems
P2 and P3.

\noindent
P2. Compute $A^r$ for $r\geq T(A)$ and $r\equiv 2(\modd\; 6)$.\\
P3. For given $x\in\R^9$, find ultimate orbit period of $A^k\otimes x$.

{\bf Solving P2.} Using the idea of Proposition \ref{blanka}, we perform $7$
squarings $A,A^2,A^4,\ldots$ to raise $A$ to the power $128>9\times 9$.
This brings us to the matrix
\begin{equation*}
A^{128}=
\begin{pmatrix}
A_{11}^{(128)} & A_{12}^{(128)} & A_{1M}^{(128)}\\[1.5 ex]
A_{21}^{(128)} & A_{22}^{(128)} & A_{2M}^{(128)}\\[1.5 ex]
A_{M1}^{(128)} & A_{M2}^{(128)} & A_{MM}^{(128)}
\end{pmatrix},
\end{equation*}
where
\begin{equation*}
A_{11}^{(128)}=
\begin{pmatrix}
0 & -1 & 0 & -1 \\
-1 & 0 & -1 & 0 \\
0 & -1 & 0 & -1 \\
-1 & 0 & -1 & 0
\end{pmatrix},\quad
A_{22}^{(128)}=
\begin{pmatrix}
-1 & -1 & 0\\
0 & -1 & -1\\
-1 & 0 & -1
\end{pmatrix},
\end{equation*}
all entries of $A_{12}^{(128)}$ and
$A_{21}^{(128)}$ are $-1$ and

\begin{align*}
A_{1M}^{(128)}=
\begin{pmatrix}
-2.5 & -1\\
-1.5 & -2\\
-2.5 & -1\\
-1.5 & -2
\end{pmatrix},\quad &
A_{2M}^{(128)}=
\begin{pmatrix}
-1.5 & -2\\
-2.5 & -2\\
-2.5 & -1
\end{pmatrix}\\
A_{M1}^{(128)}=
\begin{pmatrix}
-76 & -75.5\\
-75 & -76.5\\
-76 & -75.5\\
-75 & -76.5
\end{pmatrix}^T,\quad &
A_{M2}^{(128)}=
\begin{pmatrix}
-76 & -76.5\\
-76 & -76.5\\
-76 & -76.5
\end{pmatrix}^T
\end{align*}

We are lucky since $128\equiv 2(\modd\; 6)$, as we already have true
critical columns and rows of $A^r$.
However, the non-critical principal submatrix of $A^{128}$ is
\begin{equation*}
A^{(128)}_{MM}=
\begin{pmatrix}
-64 & -65.5\\
-65.5 & -64
\end{pmatrix}.
\end{equation*}
It can be checked that this is {\em not} the non-critical submatrix of $A^r$
that we seek (recall that $T(A)=154$). 
Hence, it remains to compute the principal
non-critical submatrix $A^{(r)}_{MM}$. 

We note that $A^{132}$ has critical rows and columns
of the spectral projector $Q(A)$, since $132$ is a multiple of 
$\gamma=6$. In $A^{132}$, the critical rows and columns
$1-4$ (in $C_1$) are the same as that of $A^{128}$, since $\gamma_1=2$
and both $128$ and $132$ are even. The critical rows $5-7$ (in $C_2$)
can be computed from those of $A^{128}$ by 
cyclic permutation $(5,6,7)\to(7,5,6)$, and
the critical rows $5-7$ can be computed by the inverse permutation
$(5,6,7)\to(6,7,5)$. This implies that all blocks in $A^{132}$
are the same as in $A^{128}$ above (in the analogous block decomposition
of $A^{132}$), 
except for
\begin{equation*}
A_{22}^{(132)}=
\begin{pmatrix}
0 & -1 & -1\\
-1 & 0 & -1\\
-1 & -1 & 0
\end{pmatrix},\quad
A_{2M}^{(132)}=
\begin{pmatrix}
-2.5 & -2\\
-2.5 & -1\\
-1.5 & -2
\end{pmatrix}.
\end{equation*}

Now the remaining non-critical submatrix of $A^r$ can be computed
using linear dependence \eqref{e:lindep}, which specifies to
\begin{equation*}
A_{\cdot k}^{(r)}=\bigoplus_{i=1}^7 a_{ik}^{(132)} A_{\cdot i}^{(128)},\quad
k=8,9.
\end{equation*}
This yields
\begin{equation*}
A_{MM}^{(r)}=
\begin{pmatrix}
-76.5 & -77\\
-78 & -76.5
\end{pmatrix}
\end{equation*}

{\bf Solving P3} We examine the orbit period of $A^kx$ for $x=x^1,x^2,x^3,x^4$,
where
\begin{equation*}
\begin{array}{c@{{}\quad{}}cccccccc}
x^1=[1 & 2 & 3 & 4 & 5 & 6 & 7 & 8 & 9],\\[1.5 ex]
x^2=[1 & 2 & 3 & 4 & 0 & 0 & 0 & 0 & 0],\\[1.5 ex]
x^3=[0 & 0 & 1 & 1 & 0 & 0 & 1 & 1 & 1],\\[1.5 ex]
x^4=[0 & 0 & 1 & 1 & 0 & 0 & 0 & 0 & 0].
\end{array}
\end{equation*}
We compute $y=A^{128}x$ for $x=x^1,x^2,x^3,x^4$:
\begin{equation*}
\begin{array}{c@{{}\quad{}}cccccccc}
y^1=A^{128}\otimes x^1=[8 & 7 & 8 & 7 & 7 & 7 & 8 &\times & \times],\\[1.5 ex]
y^2=A^{128}\otimes x^2=[3 & 4 & 3 & 4 & 3 & 3 & 3 &\times & \times],\\[1.5 ex]
y^3=A^{128}\otimes x^3=[1 & 1 & 1 & 1 & 1 & 0 & 0 &\times & \times],\\[1.5 ex]
y^4=A^{128}\otimes x^4=[1 & 1 & 1 & 1 & 0 & 0 & 0 &\times & \times].
\end{array}
\end{equation*}
Here $\times$ correspond to non-critical entries which we do not need.
The cyclic classes of $C_1$ are $\{1,3\}$, $\{2,4\}$,
and the cyclic classes of $C_2$ are $\{5\}$, $\{6\}$ and $\{7\}$.
From the considerations of Proposition \ref{blanka},
it follows that the coordinate sequences $\{(A^r x)_i,\ r\geq T(A)\}$
are  
\begin{align*}
& y_1,\; y_2,\; y_1,\;y_2,\ldots,\ \text{for $i=1,2,3,4,$}\\
& y_5,\; y_6,\; y_7,\; y_5,\; y_6,\; y_7,\ldots,\ \text{for $i=5,6,7$.}
\end{align*}
Looking at $y^1,\ldots,y^4$ above, we conclude that the orbit
of $x^1$ is of the largest possible period $6$, the orbit of $x^2$ 
is of the period $2$ (in other words, $x^2\in\Attr(A,2)$), 
the orbit of $x^3$ is of the period $3$ (i.e., $x^3\in\Attr(A,3)$),
and the orbit of $x^4$ is of the period $1$ (i.e., $x^1\in\Attr(A,1)$).

\subsection{Circulants}

Here we consider a $9\times 9$ example of
definite and visualized matrix
in {\em max-plus algebra}

\begin{equation}
\label{circ-example}
A=
\begin{pmatrix}
-8  &   0   & -1 &   -8 &   -8   & -9  &  -4   & -5 &   -1\\
    -4  &  -5  &   0  &  -2  &  -6    & 0  &  -7  &  -3  &  -9\\
    -7  &  -9   & -8  &   0  &  -8   & -4  &  -6  &  -9  & -10\\
    -8  &  -8  & -10  &  -7  &   0   & -4  &  -6  & -10   & -1\\
    -2  &  -8  &  -7  &  -4  &  -8     &0  &  -3  &  -1  & -10\\
     0  &  -1  &  -2  &  -7  & -10    &-6  &  -3   & -6   & -1\\
   -10  &  -7  &  -7  &  -7  &  -6   & -1  &  -5    & 0   & -9\\
    -8  &  -3  &  -6  &  -8  &  -6   & -8  &  -5   &-10   &  0\\
    -4  &  -3  &  -5 &   -6  &  -6  & -10  &   0  &  -6   & -9
\end{pmatrix}
\end{equation}

The critical graph of this matrix consists of two
s.c.c. comprising $6$ and $3$ nodes respectively.
They are shown in Figures \ref{twocomp} and \ref{cyclasses}, together
with their cyclic classes.

\begin{figure}[h]
\begin{tikzpicture}[shorten >=1pt,->]
  \tikzstyle{vertex}=[circle,fill=black!25,minimum size=25pt,inner sep=1pt]

\foreach \name/\angle/\text in {1/180/1, 2/120/2,
                                  3/60/3, 4/0/4, 5/-60/5, 6/-120/6}
    \node[vertex,xshift=0cm,yshift=0cm] (\name) at (\angle:1.6cm) {$\text$};
\draw (2) -- (6);
\foreach \from/\to in {1/2,2/3,3/4,4/5,5/6,6/1}
    \draw (\from) -- (\to);

 \node[vertex,xshift=5cm,yshift=0cm] (7) at (-150:1.3cm) {$7$};
    \node[vertex,xshift=5cm,yshift=0cm] (8) at (90:1.3cm) {$8$};
    \node[vertex,xshift=5cm,yshift=0cm] (9) at (-30:1.3cm) {$9$};
\foreach \from/\to in {7/8,8/9,9/7}
    \draw (\from) -- (\to);
\end{tikzpicture}
\caption{Critical graph of \eqref{circ-example}\label{twocomp}}
\end{figure}
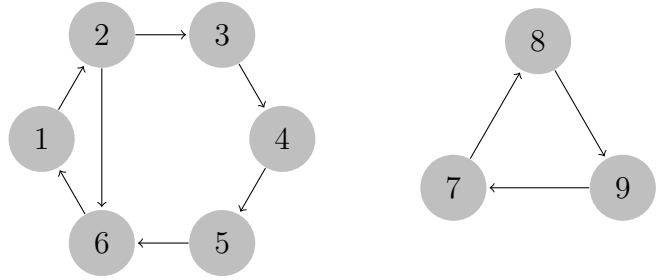

\begin{figure}[h]
\begin{tikzpicture}[shorten >=1pt,->]

\tikzstyle{vertex1}=[circle,fill=black!15,minimum size=25pt,inner sep=1pt]
\tikzstyle{vertex2}=[circle,fill=black!30,minimum size=25pt,inner sep=1pt]
\tikzstyle{vertex3}=[circle,fill=black!45,minimum size=25pt,inner sep=1pt]

\node[vertex1,xshift=0cm,yshift=0cm] (1) at (180:1.6cm) {$I$};
\node[vertex1,xshift=0cm,yshift=0cm] (4) at (0:1.6cm) {$I$};
\node[vertex2,xshift=0cm,yshift=0cm] (2) at (120:1.6cm) {$II$};
\node[vertex2,xshift=0cm,yshift=0cm] (5) at (-60:1.6cm) {$II$};
\node[vertex3,xshift=0cm,yshift=0cm] (3) at (60:1.6cm) {$III$};
\node[vertex3,xshift=0cm,yshift=0cm] (6) at (-120:1.6cm) {$III$};

\node[vertex1,xshift=5cm,yshift=0cm] (7) at (-150:1.3cm) {$IV$};
\node[vertex2,xshift=5cm,yshift=0cm] (8) at (90:1.3cm) {$V$};
\node[vertex3,xshift=5cm,yshift=0cm] (9) at (-30:1.3cm) {$VI$};

\end{tikzpicture}
\caption{Cyclic classes of the critical graph\label{cyclasses}}
\end{figure}
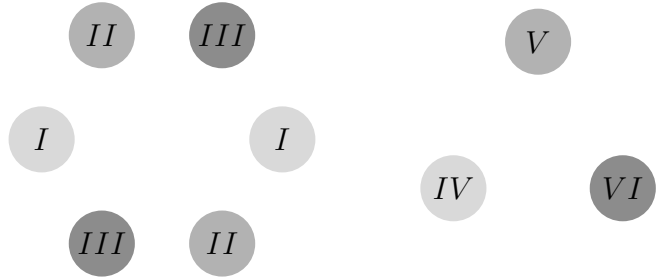

The components of $C(A)$ induce block decomposition
\begin{equation}
\label{block00}
A=
\begin{pmatrix}
A_{11} & A_{12}\\
A_{21} & A_{22}
\end{pmatrix},
\end{equation}
where
\begin{equation}
\label{a11&a22}
A_{11}=
\begin{pmatrix}
-8  &   0   & -1 &   -8 &   -8   & -9  \\
-4  &  -5  &   0  &  -2  &  -6    & 0 \\
-7  &  -9   & -8  &   0  &  -8   & -4 \\
-8  &  -8  & -10  &  -7  &   0   & -4 \\
-2  &  -8  &  -7  &  -4  &  -8     &0 \\
0  &  -1  &  -2  &  -7  & -10 & -6
\end{pmatrix},\quad
A_{22}=
\begin{pmatrix}
-5    & 0   & -9\\
-5   &-10   &  0\\
0  &  -6   & -9
\end{pmatrix}
\end{equation}
The core matrix and its Kleene star are equal to
\begin{equation}
\label{cmatxkls}
A^{\Core}=(A^{\Core})^*=
\begin{pmatrix}
0 & -1\\
-1 & 0
\end{pmatrix}.
\end{equation}
By calculating $A,A^2,\ldots$ we obtain
that the powers of $A$ become periodic after $T(A)=6$.
In the block decomposition of $A^6$ analogous to \eqref{block00},
we have the following circulants:
\begin{equation}
\label{block06}
\begin{split}
A_{11}^{(6)}=
\begin{pmatrix}
0   & -1 &   -2  &   0 &   -1  &  -2 \\
-2  &   0  &  -1  &  -2  &   0  &  -1  \\
-1  &  -2  &   0  &  -1  &  -2  &   0 \\
 0  &  -1  &  -2  &   0  &  -1  &  -2 \\
-2  &   0  &  -1  &  -2  &   0  &  -1 \\
-1  &  -2  &   0  &  -1  &  -2  &   0
\end{pmatrix},\quad &
A_{12}^{(6)}=
\begin{pmatrix}
-2 &   -1  &  -1\\
-1  &  -2  &  -1\\
-1  &  -1   & -2\\
-2  &  -1   & -1\\
-1  &  -2   & -1\\
-1  &  -1   & -2\\
\end{pmatrix},\\
A_{21}^{(6)}=
\begin{pmatrix}
-3  &  -1  &  -2  &  -3  &  -1  &  -2 \\
-2  &  -3  &  -1  &  -2  &  -3  &  -1 \\
-1  &  -2  &  -3  &  -1  &  -2  &  -3
\end{pmatrix},\quad &
A_{22}^{(6)}=
\begin{pmatrix}
0  &  -3   & -2\\
-2  &   0   & -3\\
-3  &  -2   &  0
\end{pmatrix}.
\end{split}
\end{equation}

The corresponding blocks of ``reduced'' power $\Tilde{A}^{(6)}$ are
\begin{equation}
\label{block06tilde}
\begin{split}
\Tilde{A}_{11}^{(6)}=
\begin{pmatrix}
0   & -1 &   -2 \\
-2  &   0  &  -1\\
-1  &  -2  &   0\\
 \end{pmatrix},\quad &
\Tilde{A}_{12}^{(6)}=
\begin{pmatrix}
-2 &   -1  &  -1\\
-1  &  -2  &  -1\\
-1  &  -1   & -2\\
\end{pmatrix},\\
\Tilde{A}_{21}^{(6)}=
\begin{pmatrix}
-3  &  -1  &  -2\\
-2  &  -3  &  -1\\
-1  &  -2  &  -3
\end{pmatrix},\quad &
\Tilde{A}_{22}^{(6)}=
\begin{pmatrix}
0  &  -3   & -2\\
-2  &   0   & -3\\
-3  &  -2   &  0
\end{pmatrix}.
\end{split}
\end{equation}

Note that $\Tilde{A}_{11}^{(6)}$
and $\Tilde{A}_{22}^{(6)}$ are
Kleene stars, with all off-diagonal entries negative.

Using \eqref{block06}, we
specialize system \eqref{mainsys}
to our case, we see that this system of equations for
the attraction cone $\Attr(A,1)$ consists of two chains of equations, namely

\begin{equation}
\begin{split}
x_1\oplus x_4\oplus(x_8-1)&\oplus(x_9-1)=\\
&=x_2\oplus x_5\oplus(x_7-1)\oplus(x_9-1)=x_3\oplus x_6\oplus(x_7-1)\oplus(x_8-1),\\
(x_2-1)\oplus (x_5-1)&\oplus x_7=\\
&=(x_3-1)\oplus(x_6-1)\oplus x_8=(x_1-1)\oplus(x_4-1)\oplus x_9.
\end{split}
\end{equation}

Note that only $0$ and $-1$,
the coefficients of $(A^{\Core})^*$ (which is equal to
$A^{\Core}$ in our example),
appear in this system.

\subsection{Algorithm for the strongly connected case.}

Here we consider a $6\times 6$ max-plus example
\begin{equation}
\label{algoex}
A=
\begin{pmatrix}
-3  &  0  &  -1 &  -19  &  -7 &   -3\\
-3  &  -4 &   0  & -10 &  -19  & -16\\
 0 &   -3  &  -1 &  -10  &  -8   & -8\\
-1  &  -4  &  -4  &  -1  &  -1  &  -3\\
-1  &  -1  &  -4  &  -2  &  -4  &  -1\\
-4  &  -2  &  -4  &  -1  &  -4  &  -1
\end{pmatrix},
\end{equation}
and apply to it the algorithm described in Subsect.~\ref{ss:str}.
The critical graph of this matrix consists just of one
cycle of length $3$, and there are $3$ non-critical nodes.

\begin{figure}[h]
\begin{tikzpicture}[shorten >=1pt,->]
  \tikzstyle{critical1}=[circle,fill=black!10,minimum size=17pt,inner sep=1pt]
\tikzstyle{critical2}=[circle,fill=black!20,minimum size=17pt,inner sep=1pt]
\tikzstyle{critical3}=[circle,fill=black!30,minimum size=17pt,inner sep=1pt]
\tikzstyle{noncritical}=[circle,fill=black!60,minimum size=17pt,inner sep=1pt]

    \node[critical1,xshift=0cm,yshift=0cm] (1) at (-150:1.2cm) {$1$};
    \node[critical2,xshift=0cm,yshift=0cm] (2) at (90:1.2cm) {$2$};
    \node[critical3,xshift=0cm,yshift=0cm] (3) at (-30:1.2cm) {$3$};
\node[noncritical,xshift=3cm,yshift=0cm] (4) {$4$};
\node[noncritical,xshift=4.5cm,yshift=0cm] (5) {$5$};
\node[noncritical,xshift=6cm,yshift=0cm] (6) {$6$};
\draw (1) -- (2);
\draw (2) -- (3);
\draw (3) -- (1);
\end{tikzpicture}
\caption{Critical graph and non-critical nodes of \eqref{algoex}\label{f:algoex}}
\end{figure}
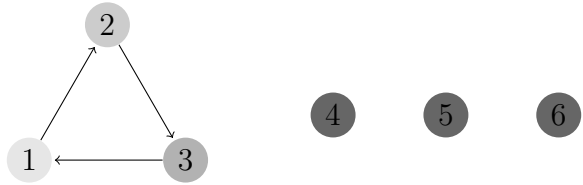

The core matrix in this case is equal to
\begin{equation*}
A^{\Core}=
\begin{pmatrix}
0 & -10 &   -7 & -3\\
-1 & -1  &  -1  &  -3\\
-1 & -2  &  -4  &  -1\\
-2 & -1  &  -4  &  -1
\end{pmatrix}
\end{equation*}

Vector $h=(-10\; -7\; -3)^T$, whose components
are computed by
\begin{equation}
\label{hexampl}
h_i=\bigoplus_{k=1}^3 a_{ki},\ \text{for $i=4,5,6$,}
\end{equation}
comprises $2,3,4$-components of the first row of $A^{\Core}$.
The arguments of maxima in \eqref{hexampl} give, after the cyclic shift by
one position, the Boolean vectors
\begin{equation}
P_4=(1\; 0\; 1),\ P_5=(0\; 1\; 0),\ P_6=(0\; 1\; 0).
\end{equation}
These vectors encode, for the corresponding non-critical nodes
$t=4,5,6,$ the starting cyclic classes (here, just critical nodes!)
of paths which go from $\crit(A)$ directly to $t$ and whose length is $3$.

The non-critical principal submatrix of $A$ and its Kleene star
are equal to
\begin{equation*}
B=
\begin{pmatrix}
 -1  &  -1  &  -3\\
 -2  &  -4  &  -1\\
 -1  &  -4  &  -1
\end{pmatrix},\quad
B^*=
\begin{pmatrix}
 0  &  -1  &  -2\\
 -2  &  0  &  -1\\
 -1  &  -2  &  0
\end{pmatrix}
\end{equation*}
The lengths of optimal non-critical paths (whose weights are entries
of $B^*$) can be written in the matrix of sets
\begin{equation}
U=
\begin{pmatrix}
\{0\} & \{1\} & \{2\}\\
\{1,2\} & \{0\} & \{1\}\\
\{1\} & \{2\} & \{0\}
\end{pmatrix}
\end{equation}

Further we compute
\begin{equation*}
h^T\otimes B^*=(-10\;-7\; -3)\otimes
\begin{pmatrix}
0 & -1 & -2\\
-2 & 0 & -1\\
-1 & -2 & 0
\end{pmatrix}=
(-4\; -5\; -3)
\end{equation*}

The maxima in $\bigoplus_t h_t b^*_{ti}$ for all $i$ are achieved
only at $t=6$, so $W_4=W_5=W_6=\{6\}$. Hence $G_4$, $G_5$ and $G_6$
are shifted $P_6$ and the shift is determined by the components in
the last row of $U$ which is $(\{1\}\; \{2\}\; \{0\})$. From
$P_6=(0\; 1\; 0)$ we conclude that
\begin{equation*}
G_4=(0\; 0\; 1),\ G_5=(1\; 0\; 0),\ G_6=(0\; 1\; 0).
\end{equation*}
By Proposition \ref{p:gti} we have that $G_i(t)=1$ if and only
if $t\in K^{(r)}(i)$ (where $r\geq T(A)$ is a multiple
of $\caa=3$). Using this rule we obtain that
$K^{(r)}(1)=\{5\}$, $K^{(r)}(2)=\{6\}$, $K^{(r)}(3)=\{4\}$,
and using the vector of coefficients
$h^T\otimes B^*=(-4\; -5\; -3)$,
we can write out the system for attraction cone
\begin{equation}
\label{attrsysex}
x_1\oplus(x_5-5)=x_2\oplus(x_6-3)=x_3\oplus(x_4-4).
\end{equation}

To verify this result, we observe
that in our case $T(A)=8$ and
\begin{equation*}
\begin{split}
A^8&=
\begin{pmatrix}
-1 & -1 & 0 & -4 & -6 & -4\\
0 & -1 & -1 & -5 & -5 & -4\\
-1 & 0 & -1 & -5 & -6 & -3\\
-2 & -1 & -2 & -6 & -1 & -4\\
-2 & -1 & -1 & -5 & -7 & -4\\
-2 & -3 & -2 & -6 & -7 & -6
\end{pmatrix}\\
A^9&=
\begin{pmatrix}
0 & -1 & -1 & -5 & -5 & -4\\
-1 & 0 & -1 & -5 & -6 & -3\\
-1 & -1 & 0 & -4 & -6 & -4\\
-2 & -2 & -1 & -5 & -7 & -5\\
-1 & -2 & -1 & -5 & -6 & -5\\
-2 & -2 & -3 & -7 & -7 & -5
\end{pmatrix}\\
\end{split}
\end{equation*}

Applying cancellation to the critical subsystem of
$A^8\otimes x=A^9\otimes x$, we obtain \eqref{attrsysex}.

\section{Acknowledgement}

The author is grateful to Peter Butkovi\v{c} for
valuable discussions and comments concerning this work,
and also to Hans Schneider for sharing his
original point of view on nonnegative matrix scaling
and max algebra. The author also wishes to thank
Trivikram Dokka for sharing his helpful experience in attraction cones
\cite{Dok:08}.

\end{document}